\documentclass[11pt, a4paper]{amsart}
\usepackage{amsmath}
\usepackage{mathdots}

\newtheorem{theorem}{Theorem}[section]     

\newtheorem{lemma}[theorem]{Lemma}
\newtheorem{proposition}[theorem]{Proposition}
\newtheorem{corollary}[theorem]{Corollary}
\newtheorem{remark}[theorem]{Remark} 

\usepackage[colorlinks=true,citecolor=cyan,backref=page]{hyperref}

\usepackage[vcentermath]{youngtab}
\usepackage{ytableau}
\usepackage[all]{xy}

\usepackage{xspace}

\usepackage{shuffle}
\usepackage{ulem}

\DeclareMathOperator{\Plax}{Plax}
\DeclareMathOperator{\Styl}{Styl}
\DeclareMathOperator{\Supp}{Supp}
\DeclareMathOperator{\Part}{Part}
\DeclareMathOperator{\evac}{evac}
\DeclareMathOperator{\rk}{rk}

\newcommand{\N}{\mathbb N}

\newcommand{\C}{\mathcal C}
\newcommand{\p}{\mathbb P}
\newcommand{\D}{\mathbf D}

\usepackage{tikz}
\usetikzlibrary{patterns}
\usetikzlibrary{fadings}
\tikzfading[name=fade out, inner color=transparent!0, outer color=transparent!90]
\usepackage[labelfont=normalfont,labelformat=simple]{subcaption}
\usepackage{graphicx}

\usepackage[vcentermath]{youngtab}
\usepackage{ytableau}

\usepackage{enumerate}

\title{The stylic monoid}

\author{A. Abram}
\address{Antoine Abram,
D\'epartement de math\'ematiques, Universit\'e du Qu\'ebec \`a Montr\'eal}
\email{abram.antoine@courrier.uqam.ca}

\author{C. Reutenauer}
\address{Christophe Reutenauer,
D\'epartement de math\'ematiques, Universit\'e du Qu\'ebec \`a Montr\'eal}
\email{Reutenauer.Christophe@uqam.ca}

\date{\today}

\begin{document}

\keywords{plactic monoid, stylic monoid, tableaux, partitions, standard immaculate tableaux, evacuation, $J$-trivial, $J$-order}

\begin{abstract} The free monoid $A^*$ on a finite totally ordered alphabet $A$ acts at the left on columns, by Schensted left insertion. This defines a finite monoid, denoted $\Styl(A)$ and called the stylic monoid. It is canonically a quotient of the plactic monoid. Main results are: the cardinality of $\Styl(A)$ is equal to the number of partitions of a set on $|A|+1$ elements. We give a bijection with so-called $N$-tableaux, similar to Schensted's algorithm, explaining this fact. Presentation of $\Styl(A)$: it is generated by $A$ subject to the plactic (Knuth) relations and the idempotent relations $a^2=a$, $a\in A$. The canonical involutive anti-automorphism on $A^*$, which reverses the order on $A$, induces an involution of $\Styl(A)$, which similarly to the corresponding involution of the plactic monoid, may be computed by an evacuation-like operation (Schützenberger involution on tableaux) on so-called standard immaculate tableaux (which are in bijection with partitions). The monoid $\Styl(A)$ is $J$-trivial, and the $J$-order of $\Styl(A)$ is graded: the co-rank is given by the number of elements in the $N$-tableau. The monoid $\Styl(A)$ is the syntactic monoid for the the function which associates to each word $w\in A^*$ the length of its longest strictly decreasing subword.
\end{abstract}

\maketitle

\subjclass{20M05, 20M20, 20C30, 05E10, 06F05}

\tableofcontents

\section{Introduction}
The plactic monoid is a fundamental object in combinatorics, representation theory, and algebra. It originates in a bijection 
of Schensted \cite{S}, often called the {\it Robinson-Schensted-Knuth correspondence}.
Let $A$ be a totally ordered finite alphabet $A$. The Schensted bijection maps each word $w\in A^*$ onto a 
pair $(P(w),Q(w))$, where $P(w)$ is a semi-standard Young tableau on $A$, and $Q(w)$ is standard Young tableau on $\{1,2,\ldots, n\}$ ($n$ is 
the length of $w$), both tableaux having the same shape.
It turns out 
that the condition $P(w)=P(w')$ defines a congruence on the free monoid $A^*$. This congruence was called the {\it plactic 
congruence} by Lascoux and Sch\"utzenberger, and they studied in \cite{LS1} the corresponding quotient monoid $A^*/
{\equiv_{plax}}$, called the {\it plactic monoid}. This monoid has a cubic presentation, given by Knuth \cite{K}, with set of 
generators $A$, and relations called 
the {\it plactic relations}. A survey on the plactic monoid and its applications is given by 
Lascoux, Leclerc and Thibon (Chapter 5 of Lothaire's book \cite{LLT}).

The plactic monoid has another natural finite generating set, the set of {\it columns}. A column is a strictly decreasing word.
With this generating set, it has a quadratic presentation, which turns out to be confluent \cite{BCCL,CGM} (note that the 
standard presentation is not confluent \cite{O}). 

Columns play a special role in the plactic monoid, which may be very deep as is seen in the first section of \cite{LS2}. 
Clearly, the first column of $P(w)$ depends only on the plactic class of $w$. In that way, one obtains by left multiplication 
an action of the plactic monoid on the finite set of columns. We call {\it stylic monoid} the finite monoid of endofunctions of 
this set obtained by this action (for the teminology, we use the Greek word for columns). Clearly, this monoid is a finite 
quotient of the plactic monoid.

Note that in the literature, one finds a class of monoids called {\it partition monoids}, see \cite{HR}. They are related to the 
Temperley-Lieb algebra, and different from the stylic monoids.

The two first main results give the cardinality of this monoid, and a presentation of it (Theorem \ref{main}). 
Let $n$ be the cardinality of $A$. Then the cardinality of $\Styl(A)$ (the stylic monoid on $A$) is equal to the number of partitions of a set with $n+1$ 
elements, the Bell number $B_{n+1}$. Moreover, the presentation on the set of generators $A$ is obtained by adding to 
the plactic relations the {\it idempotent relations} $a^2=a$ for each generator $a\in A$.

In course of the proof, we establish a bijection between  $\Styl(A)$ and a set of semi-standard tableaux that we call {\it 
$N$-tableaux}\,: they are obtained by the condition that the rows strictly increase, and that each row contains the next one.
The bijection is a variant of Schensted right insertion.

Next, we study a natural involution on $\Styl(A)$. It is obtained from the anti-automorphism $\theta$ of the free monoid 
$A^*$ which reverses words, and reverses the alphabet (for example, $123233\mapsto 112123$, $A=\{1,2,3\}$). It induces 
an anti-automorphism of both the plactic monoid and the stylic monoid, as is seen on the plactic relations and idempotent 
relations. Concerning the plactic monoid, there is a remarkable direct construction on tableaux of this involution by 
Sch\"utzenberger, called {\it evacuation}.

This leads us to a similar construction for the stylic monoid. First, it is easy to see that $N$-tableaux are bijectively 
represented by partitions of subsets of $A$. Such a partition may be represented by an increasing labeling of a lower ideal 
of $\p^2$, the latter being ordered as is shown in Figure \ref{order-preceq}. This allows to mimick the classical theory for 
standard tableaux: tableaux, skew-tableaux, jeu de taquin, evacuation. The third main result is that this modified 
evacuation corresponds to the involution (Theorem \ref{evac}). The proof is nontrivial, but we followed the classical case (skew diagrams with a hole \cite{Sc2}), 
as is shown in Sagan's book \cite{S}, with the help of Fomin's growth diagrams, which may be extended to our case: 
partitions are replaced by compositions, appropriately ordered. We use a notion that appeared previously in the literature: 
composition tableaux of \cite{HLMvW,LMvW} (with one condition removed), and more precisely, standard immaculate tableaux \cite{BBSSZ} (see also \cite{BBSSZ2}, \cite{C},  \cite{G}, \cite{AHM}, and \cite{NTT}).

Next, we prove a semigroup-theoretical property of the stylic monoid: it is $J$-trivial. This follows from the action on 
columns, and its order properties, once columns are naturally ordered. It is well-known that $J$-trivial monoids inherit the 
$J$-order: $x\leq_Jy$ if $x$ is in the two-sided ideal generated by $y$. The fourth main result is that in the stylic monoid, 
the $J$-order is graded (Theorem \ref{graded}). For the proof of this, we define the left insertion of a letter in an $N$-tableau, which corresponds 
to multiplication at the left in the monoid. Unlike Schensted left and right insertion, which are symmetric, the left and right 
insertion into $N$-tableaux are completely asymmetric. The $J$-order of the stylic monoid induces an order on set 
partitions, which seems new; in particular, the height of this graded poset is quadratic (unlike the usual refinement order of partitions, whose height is linear).

The fifth main result is an automata-theoretic result: the stylic monoid is syntactic with respect to the function which associates to each word the length of its longest strictly decreasing subsequence, equivalently  by Schensted's theorem, the length of the first column of its $P$-tableau (Theorem \ref{syntac}). 

We extend the methods to prove this result to give, in the Appendix, a proof of a statement given without proof by Lascoux and Sch\"utzenberger \cite{LS1}: the plactic monoid is syntactic with respect to the function which associates to each word the shape of its $P$-tableau (Theorem \ref{plactic-synt}).

We give also some order-theoretic properties of the action on columns, and as an application, a new proof of the quadratic presentation of the plactic monoid generated by columns, mentioned at the beginning of the introduction (Theorem \ref{quadr-pres}, due to \cite{BCCL,CGM}).

{\bf A remark} about terminology, notations and abuse of language: a word $a_1\cdots a_n, a_i\in A$, where $A$ is a totally ordered alphabet, is called {\it increasing} (resp. {\it strictly increasing}) if $a_1\leq \cdots\leq a_n$ (resp. $a_1< \cdots<a_n$). Similarly for {\it decreasing}.

We use the notion of {\it columns}, which are 
considered simultaneously as Young tableaux, as subsets of $A$, and as strictly decreasing words on $A$. 
We find this more convenient than introducing three different notations.

\section{Schensted insertions}\label{insert}

Let $A$ be a totally ordered finite {\it alphabet} (whose elements are called {\it letters}) and denote by $A^*$ the set of {\it words} on $A$, which is the {\it free 
monoid} freely generated by $A$.

In this article, we call a {\it tableau} what is called usually a {\it semi-standard Young tableau}; that is, a finite lower order ideal (that is, a finite subset $E\subset \N^2$ such that $x\leq y$ and $y\in E$ implies $x\in E$)
of the poset $\N^2$, ordered naturally, together with an increasing mapping into $A$, such that the restriction of this mapping to 
each subset with given $x$-coordinate is injective. A tableau is usually represented as in Figure \ref{tab}. The 
conditions may be expressed by saying that the letters in $A$ are weakly increasing from left to right in each row, and 
strictly 
increasing from the bottom to top in each column.

We call {\it support} of a word $w$, and denote it by $\Supp(w)$, the set of letters appearing in $w$. Similarly for the 
support of a tableau, denoted likewise.

Call {\it column} a tableau with only one column, and {\it row} a tableau with only one row. 
One may see a column
as a subset of $A$, and a row as a multiset of elements of $A$. We shall use therefore the symbol $\cup$ to express 
union of columns, and of rows (for rows, it is the multiset union). The empty column (resp. row) is denoted by $\emptyset$.

Another useful way to view columns is as {\it decreasing word} (a word whose letters decrease strictly from left to right).

We define now the {\it column insertion}. Let $\gamma$ be a column, viewed here as a subset 
of $A$, and let $x\in A$. There are two cases: if $\forall y\in\gamma, x>y$, then define $\gamma'= 
\gamma\cup x$. Otherwise, let $y$ be the smallest element in $\gamma$ with $y\geq x$; then define $
\gamma'=(\gamma\setminus y)\cup x$. Then $\gamma'$ is the column obtained by {\it column insertion of $x$ into} $
\gamma$, and in the second case, $y$ is said to be {\it bumped}.

%
%
%

One defines the {\it  column insertion of $x\in A$ into a tableau} $T$ recursively as follows: insert $x$ 
into the first column (the leftmost); in the case no element is bumped, stop; otherwise insert the bumped element in the 
second column, and so on.

Finally, given a word $w=a_1\cdots a_n$ on $A$, and a tableau $T$, one defines the  {\it column insertion of 
$w$ into} $T$ recursively by inserting $a_n$ into $T$, then $a_{n-1}$ into the tableau obtained, and so on.

The {\it  insertion into a row of} $x\in A$ is defined similarly: exchange $>$ and $\geq$ in the definition of the column insertion (for a multiset $E$ containing $y$, $E\setminus y$ means that one $y$ is removed from $E$).

The {\it  row insertion in a tableau} is defined similarly to column insertion, by using row insertion and starting from the 
first row (the one with $y$-coordinate 0).
 
Similarly, the {\it row insertion of a word $w$ into} $T$ is obtained recursively by row insertions, starting with $a_1$, then $a_2$ 
and so on.

\begin{figure}
\begin{ytableau}
 d\\
b&  b\\
a&a&c \end{ytableau}
\caption{A tableau}\label{tab}
\end{figure}

A fundamental result of Schensted \cite{S} is that inserting a word $w$ into the empty tableau gives the same 
tableau, by column insertion, or by row insertion. The resulting tableau is denoted by $P(w)$. See \cite{S}, or \cite[Chapter 3]{Sa}, for details.

It follows that for each words $u,v$,
$P(uv)$ is equal to the tableau obtained by column insertion of $u$ into $P(v)$, and also by row insertion of $v$ into $P(u)$.

Another fundamental result of Schensted \cite{S} states that the maximal length of a strictly decreasing subsequence 
of the word $w$ is equal to 
the number of rows of the tableau $P(w)$. Similarly, the maximal length of a weakly increasing subsequence of $w$ is 
equal to the number of columns of $P(w)$.

\section{The plactic monoid}\label{plactic}

The condition $P(u)=P(v)$ is a monoid congruence  on the free monoid, as follows from the previous section. This 
congruence was called the {\it plactic congruence}, denoted $\equiv_{plax}$, and the quotient monoid $\Plax(A)$ was 
called the {\it plactic monoid} by Lascoux and Sch\"utzenberger \cite{LS1}. It follows from the work of Knuth \cite{K} 
that the plactic congruence is generated by the relations
$$
bac \equiv_{plax} bca, \quad acb\equiv_{plax}cab, \quad baa\equiv_{plax}aba, \quad bba\equiv_{plax}bab,
$$
for all choices of letters $a<b<c$ in the first two relations, and for all choices of letters $a<b$ in the two others.

By definition, the plactic monoid may be identified with the set of tableaux on $A$, and the surjective monoid homomorphism from $A^*$ into $Plax(A)$ is therefore denoted $P$.

Define for each tableau $T$ its {\it row-word} to be the word, denoted $RW(T)$, obtained by reading its rows from left to right, starting with the row of largest 
$y$-coordinate; for example the row-word of the tableau in Figure \ref{tab} is $dbbaac$. Similarly, its column-word, denoted $CW(T)$, is obtained by 
reading the columns from left to right, each column being read by starting with the box with highest $y$-coordinate; in the figure, it is $dbabac$.

In particular, the row-word of a column $\gamma$ is a strictly decreasing word, equal to its column-word. We often identify $\gamma$ with this word.

It is a well-known result that for each tableau $T$, one has

$$
T=P(RW(T))=P(CW(T)),
$$
and thus
$$
RW(T)\equiv_{plax}CW(T).
$$
Moreover, for any word $u$,
$$
u\equiv_{plax}RW(P(u)).$$
See \cite[Lemma 3.6.5]{Sa}, \cite[Th. A1.1.6]{St}, \cite[Theorem 5.2.5 and Problem  5.2.4]{LLT}.

\section{An action on columns}\label{Action}

Denote by $\mathcal C(A)$ the set of columns on $A$. We define a left action of $A^*$ on $\mathcal C(A)$, denoted $u\cdot \gamma$, for each $u\in A^*$ and each 
column $\gamma$. Since $A^*$ is the free monoid on $A$, it is enough to define the action for each letter $a\in A$. 
Define 
$$a\cdot \gamma=\gamma'$$
if $\gamma'$ is obtained from $\gamma$ by  column insertion of $a$ into $\gamma$.

\begin{proposition}\label{first}
Let $\gamma$ be a column and $w$ be a word. Then $w\cdot \gamma$ is the first column of $P(wRW(\gamma))$, which is obtained by  row insertion of 
$RW(\gamma)$ into $P(w)$.
\end{proposition}

\begin{proof} 
$P(wRW(\gamma))$ is the tableau obtained by  column insertion of $w$ into $P(RW(\gamma))=\gamma$ (see Section \ref{insert}). It follows from the definitions of column insertion and the action on columns that its first column is precisely $w\cdot\gamma$. But $P(wRW(\gamma))$ is also the tableau obtained by row insertion of $RW(\gamma)$ into $P(w)$, see Section \ref{insert}.
\end{proof}

For a column $\gamma$, and a letter $x$, define $\gamma_x=\{y\in\gamma \mid y<x\}$ and $\gamma^x=\{y\in\gamma \mid y>x\}$.

\begin{lemma}\label{idemp} Let $\gamma$ be a column and $x$ be a letter.

(o) $x\cdot \gamma$ contains $x$.

(i) If $\gamma$ contains $x$, then $x\cdot\gamma=\gamma$. 

(ii) One has $(x\cdot \gamma)_x=\gamma_x$.
\end{lemma}

\begin{proof} All these statements follows from the definition of the insertion of a letter in a column. 
\end{proof}

\begin{corollary}\label{recursive-action} Let $\gamma$ be a column and $w$ be a word.

(i) If $\Supp(w)\subseteq \gamma$, then $w\cdot\gamma =\gamma$.

(ii) Let $\ell $ be a letter and $B=\{x\in A\mid x\leq \ell\}$. If $B\subseteq \gamma$, then $B\subseteq w\cdot\gamma$.
\end{corollary}

\begin{proof} (i) follows from Lemma \ref{idemp} (i) by induction on the length of $w$. For (ii), we argue also by induction. The case when $w$ is empty is clear. Suppose that $w=xu$, $x\in A$, $u\in A^*$. Then $B\subseteq u\cdot \gamma=\gamma'$ by induction. We have $w\cdot\gamma=x\cdot\gamma'$. If $x\leq \ell$, then $x\in B\subseteq \gamma'$, hence $x\cdot\gamma'=\gamma'$ by Lemma \ref{idemp} (i) and consequently $B\subseteq x\cdot\gamma'$. If $x>\ell$, then $B\subseteq \gamma'_x$; since $(x\cdot\gamma')_x=\gamma'_x$ by Lemma \ref{idemp} (ii), we have 
$B\subseteq x\cdot\gamma'$.
\end{proof}

 \section{The stylic monoid}
 
We denote by $\Styl(A)$ the monoid of endofunctions of the set $\mathcal C(A)$ of columns obtained by the action defined in the previous 
section. Since $\mathcal C(A)$ is finite, $\Styl(A)$ is finite. Let $\mu:A^*\to \Styl(A)$ be the canonical monoid homomorphism. We denote by $\equiv_{styl}$ the 
corresponding monoid congruence of $A^*$: $u\equiv_{styl} v$, if and only if $\mu(u)=\mu(v)$, if and only if for each 
column $\gamma$, $u\cdot \gamma=v\cdot \gamma$. The monoid $\Styl(A)$ acts naturally on the set of columns, and we 
take the same notation: $m\cdot \gamma=w\cdot \gamma$ if $m=\mu(w)$.

\begin{proposition}\label{plax-styl} If $P(u)=P(v)$, then for any column $\gamma$, $u\cdot \gamma=v\cdot \gamma$, and in particular, 
$u\equiv_{styl} v$. Thus $\Styl(A)$ is naturally a quotient of $\Plax(A)$: $u\equiv_{plax}v\Rightarrow u\equiv_{styl}v$.
\end{proposition}

\begin{proof} By Proposition \ref{first}, $u\cdot \gamma$ is the first column of $P(uRW(\gamma))$; the latter element of $\Plax(A)$ is equal to $P(u)P(RW(\gamma))=P(v)P(RW(\gamma))=P(vRW(\gamma))$, whose first column is by the same result equal to $v\cdot \gamma$. 
\end{proof}

\begin{lemma}\label{idemp2}
For $x\in A$, $x^2\equiv_{styl}x$.
\end{lemma} 

\begin{proof} 
This follows from Lemma \ref{idemp} (o) and (i).
\end{proof}

Note that one has for any $u\in A^*$:
$$
u\equiv_{styl} RW(P(u)),
$$
since $u\equiv_{plax} RW(P(u))$ (Section \ref{plactic}).

It follows that for each element $m=\mu(u)$ of $\Styl(A)$, one has $m=\mu(RW(P(u)))$. Take $u$ of smallest length. Then no row of $P(u)$ contains repeated elements, otherwise $RW(P(u))$ contains a factor $aa$, and by Lemma \ref{idemp2}, $RW(P(u))\equiv_{styl} v$ for some word of shorter length. 

Hence each element of $\Styl(A)$ is represented by a tableau which has strictly increasing rows (and columns are evidently strictly increasing, too).

We note that this set of tableaux is not bijectively mapped onto $\Styl(A)$ (only surjectively). Indeed, an example of two such distinct tableaux which are mapped onto the same element of $\Styl(A)$ are shown in Figure \ref{cabd} and \ref{cdab}.
\begin{figure}\centering
\begin{minipage}[b]{0.45\linewidth}\centering
\begin{ytableau}
c\\
a&b&  d
\end{ytableau}
\caption{}\label{cabd}
\end{minipage}
\hfill
\begin{minipage}[b]{0.45\linewidth}\centering
\begin{ytableau}
c&d\\
a&b
\end{ytableau}
\caption{}\label{cdab}
\end{minipage}
\end{figure}
Their row words are equal modulo $\equiv_{styl}$, since we have the sequence of equivalences, using only the plactic congruence and the relation $cc\equiv_{styl}c$ (Lemma \ref{idemp2}):
$\underline{c}abd\equiv_{styl} c\underline{cab}bd \equiv_{styl} ca\underline{cbd} \equiv_{styl} \underline{cac}db  \equiv_{styl} cc\underline{adb} \equiv_{styl}  \underline{ cc}dab \equiv_{styl}  
cdab$, where underlines indicate the left-hand side of the relation which is used.

For further use, we state the following lemma.

\begin{lemma}\label{support}
If two words $u$ and $v$ have the same action on the set of columns over $A=\Supp(u)\cup \Supp(v)$, then $\Supp(u)=\Supp(v)$.
\end{lemma}

It follows that the function $\Supp$ is well-defined on $\Styl(A)$ (this will be also a consequence of Theorem \ref{fixed}).

\begin{proof} Suppose that $\Supp(u)\neq \Supp(v)$. By symmetry, we may assume that there exists a letter $\ell$ such that $\ell\in\Supp(u),\ell\notin \Supp(v)$. Define the 
column $\gamma=A\setminus \ell$. Then, $\Supp(v)\subseteq \gamma$, hence  $v\cdot \gamma=\gamma$ by Corollary \ref{recursive-action} (i), and in particular $\ell\notin v\cdot \gamma$. We may write $u=u_1\ell u_2$, where $\ell\notin\Supp(u_2)$; then $u_2\cdot 
\gamma=\gamma$ by Corollary \ref{recursive-action}  (i); next,  $\ell\cdot \gamma=\gamma'$, where $\gamma'$ has the property that it contains all the letters less than or equal to $\ell$; hence, $u_1\cdot 
\gamma'$ has also this property, by Corollary \ref{recursive-action} (ii). Since $u\cdot\gamma=u_1\cdot 
\gamma'$, we have $u\cdot\gamma\neq v\cdot\gamma$, and $u,v$ are not equivalent modulo $\equiv_{styl}$.
\end{proof}

\begin{proposition}\label{zero} The monoid $\Styl(A)$ has a zero, which is the image under $\mu$ of the decreasing product of all letters in $A$.
\end{proposition}

\begin{proof}\footnote{We are indebted to the anonymous referee for pointing out an error in an earlier version of this proof.}  Let $w$ be this product, which we view also as column, denoted $\gamma_0$: it is the maximal column on $A$ for the inclusion order. We claim that for any column $\gamma$ on $A$, $w\cdot\gamma=\gamma_0$. Hence, for any letter $x$, $wx\cdot\gamma=w\cdot(x\cdot\gamma)=\gamma_0=w\cdot\gamma$; thus $wx\equiv_{styl} w$. Moreover, $xw\cdot\gamma=x\cdot(w\cdot\gamma)=x\cdot\gamma_0=\gamma_0=w\cdot \gamma$; thus $xw\equiv_{styl}w$. Therefore $w$ is the zero of the stylic monoid.

We prove now the claim. Let $x$ any letter; then $w=uxv$ and: ($*$) each letter in $u$ is greater than $x$. By Lemma \ref{idemp} (o), $(xv)\cdot \gamma=x\cdot(v\cdot\gamma)$ contains $x$. Then, an easy induction on the length of $u$, using ($*$) and Lemma \ref{idemp} (ii), implies that $u\cdot ((xv)\cdot \gamma)$ also contains $x$. Hence $w\cdot \gamma$ contains $x$. Thus $w\cdot\gamma$ contains $A$, and finally $w\cdot\gamma=\gamma_0$.
\end{proof}

\section{A variant of Schensted row insertion}\label{N-tab}

\subsection{$N$-tableaux and right $N$-insertion}

Define an \textit{$N$-tableau} to be a tableau satisfying the following two conditions:

(i) the rows are strictly increasing;

(ii) each row is contained in the row below.

Note that the support of an $N$-tableau coincides with its first row.
As an example, see Figure $\ref{Tableau}$.  

\begin{figure}
\begin{ytableau}
d&e\\
b&d &e\\
a & b& c&d&e
\end{ytableau}
\caption{An $N$-tableau}\label{Tableau}
\end{figure}

To each $N$-tableau whose support is $A_1\subseteq A$, associate  the decreasing sequence of subsets of $A_1$ 
\begin{equation}\label{decreasing}
A_1 \supseteq A_2 \supseteq   A_3 \ldots
\end{equation}
where $A_i$ is the $i$-th row, viewed as a set. One has 
\begin{equation}\label{minima}
\min(A_1)<\min(A_2)<\min(A_3)\ldots,
\end{equation} 
since these elements constitute the first column of the $N$-tableau. We call {\it $N$-filtration on $A_1$} a sequence of subsets of $A_1$ satisfying (\ref{decreasing}) and (\ref{minima}); when $A_1$ is understood, we also say simply {\it $N$-filtration}. Note that the condition on the minima implies that the sequence is {\it strictly decreasing}.

Conversely, given an $N$-filtration, one associates with it an $N$-tableau, as is easily verified. Therefore, $N$-tableaux and $N$-filtrations are in bijection.

We describe now an algorithm, called the {\it right} $N$-{\it algorithm}, which associates with each word $w \in A^*$ an $N$-tableau $N(w)$. Viewing strictly increasing rows as subsets of 
$A$, let $B\subseteq A$ be such a row. The {\it right} $N$-{\it insertion of a letter $x$ in} $B$ is equal to $B\cup x$, and 
 if $y$ is the smallest element of $B$ which is strictly greater than $x$, then a copy of $y$ is bumped (and $y$ does not disappear from $B$). 
Note  that no element is bumped if and only if $x$ is greater than or equal to the elements of $B$.

Now {\it right} $N$-{\it insertion of $x$ in an} $N$-tableau is recursively defined as for the Schensted row insertion: insert $x$ in the 
first row, then the bumped element, if any, in the second one, and so on. 
For an example of this, see Figure \ref{insertion}.

\begin{figure}
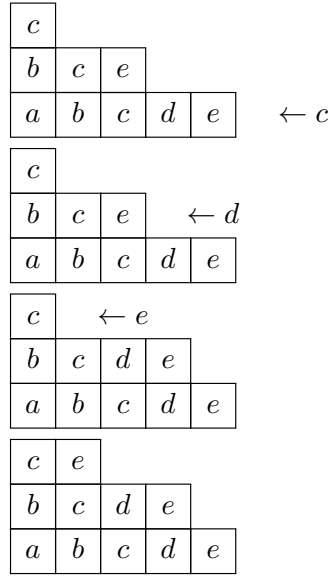

\begin{ytableau}
c\\
b&c&e\\
a&b&c&d&e&\none &\none[\leftarrow c]\\
\end{ytableau}
\vspace{1mm}

\begin{ytableau}
c\\
b&c&e&\none &\none[\leftarrow d]\\
a&b&c&d&e&\none &\none\\
\end{ytableau}
\vspace{1mm}

\begin{ytableau}
c&\none &\none[\leftarrow e]\\
b&c&d&e\\
a&b&c&d&e&\none &\none\\
\end{ytableau}
\vspace{1mm}

\begin{ytableau}
c&e\\
b&c&d&e\\
a&b&c&d&e&\none &\none\\
\end{ytableau}
\vspace{1mm}
\caption{Right $N$-insertion of $c$ into an $N$-tableau}\label{insertion}
\end{figure}

\begin{proposition} The right $N$-insertion of $x$ in an $N$-tableau produces an $N$-tableau.
\end{proposition}
 
If $T$ is an $N$-tableau, we denote by $T\leftarrow x$ the $N$-tableau obtained by right $N$-insertion of $x$ into $T$. 

We use in the proof below the fact that if $S$ is a tableau, with $S'$ the tableau obtained by removing the first row of $S$, assuming that $S'$ is nonempty, then 
$S$ is an $N$-tableau if and only if the three following conditions are satisfied: $S'$ is an $N$-tableau; $\min(S)<\min(S')$; $\Supp(S)\supseteq \Supp(S')$. 

\begin{proof} If in the  $N$-insertion $T\leftarrow x$, no letter is bumped, then $x$ is greater than or equal to any letter in $T$. Then $(T\leftarrow x)=T$ if $x\in T$, and otherwise $T\leftarrow x$ is obtained by adding $x$ at the end of the first row of  $T$. Thus $T\leftarrow x$ is clearly an $N$-tableau.

Otherwise, $y$ is bumped from the first row. Let $T'$ be the $N$-tableau obtained by removing the first row of $T$. Then the tableau obtained by removing the first row of $T\leftarrow x$ is the tableau $T'\leftarrow y$. This latter tableau is by induction an $N$-tableau.  By the criterion stated before the proof, it is therefore enough to show that $\min(T\leftarrow x)<\min(T'\leftarrow y)$ and that $\Supp(T\leftarrow x)\supseteq \Supp (T'\leftarrow y)$. 

We have $\Supp(T')\subseteq\Supp(T)$, $y\in \Supp(T)$, $\Supp(T\leftarrow x)=\Supp(T)\cup x$ and $\Supp(T'\leftarrow y)=\Supp(T')\cup y$; thus $\Supp(T'\leftarrow y)\subseteq \Supp (T\leftarrow x)$.

We have $\min(T\leftarrow x)=\min(\min(T),x)$ and similarly $\min(T'\leftarrow y)=\min(\min(T'),y)$. Moreover, $\min(T)<\min(T')$ and $x<y$. Thus $\min(T\leftarrow x)<\min(T'\leftarrow y)$ (since $a<a',b<b'$ implies $\min(a,b)<\min(a',b'))$.
\end{proof}

Similarly to Schensted row insertion, the {\it right $N$-insertion of a word $w$ into an $N$-tableau} $T$ is obtained by inserting the first letter of $w$ into $T$, then the second one,
and so on. We denote by $N(w)$ the $N$-tableau obtained by inserting the word $w$ into the empty $N$-tableau.

\subsection{Inflation and simulation by Schensted row insertion}

Define an {\it inflation} of a word $w=a_1\cdots a_n, a_i\in A$, to be any word 
of the form $a_1^{x_1}\cdots a_n^{x_n}$ for some positive exponents $x_i\in\N$.

We show that the right $N$-algorithm may be simulated by the Schensted row insertion algorithm, in the following sense.

\begin{lemma}\label{infl} Each word $w$ has an inflation $w'$ such that $N(w)$ and $P(w')$ have the same number of rows, and that corresponding rows in $N(w)$ and $P(w')$ have the same support.
\end{lemma}

An example will be useful to understand the lemma: the two row-words of the tableaux in Figures \ref{cabd} and \ref{cdab} are $cabd$ and $cdab$. They have the same $N$-tableau under the $N$-algorithm, namely the tableau shown in Figure \ref{cabcd}.

\begin{figure}\centering
\begin{minipage}[b]{0.45\linewidth}\centering
\begin{ytableau}
c\\
a&b&c&d
\end{ytableau}
\caption{}\label{cabcd}
\end{minipage}
\hfill
\begin{minipage}[b]{0.45\linewidth}\centering
\begin{ytableau}
c&c\\
a&b&c&d
\end{ytableau}
\caption{}\label{ccabcd}
\end{minipage}
\end{figure}


Consider $w'=c^3dab$, which is an inflation of  $w=cdab$. Then it is easily verified that $P(w')$ is equal to the tableau shown in Figure \ref{ccabcd}. The corresponding rows of $N(w)$ and $P(w')$ have the same support.


\begin{proof}[Proof of Lemma \ref{infl}] We consider the following equivalent version of Schensted row insertion of a word $w$ into a tableau $T$. For a word $w$, factorized as $w=u_1\cdots u_k$, one may insert first $u_1$ in the first row of $T$, constructing from left to right the word $v_1$ of bumped letters; then insert $v_1$ into the second row, and so on until the last row; then continue with the second factor $u_2$, and so on. We call this {\it row insertion by factors}.

It may be that each factor $u_i$ is a power of some letter, and also that each bumped word, $v_1$ and the others, are powers of some letter (not the same letter for all these words). In this case, we say that the insertion by factors {\it satisfies the block condition}. In order to be such, the necessary and sufficient condition is that each inserted factor is a power $a^i$ and that, when inserted in a row, and if letters are bumped, there must be in this row at least $i$ letters $b$, with $b$ being the minimum of the letters greater than $a$ in the row. Note that the bumped word is then $b^i$, with the same exponent.

Let $w=a_1\cdots a_n$. We show that for some choice of the exponents $x_i$, the row insertion by factors of $w'=a_1^{x_1}\cdots a_n^{x_n}$, with the factors $a_i^{x_i}$, satisfies the block condition.

Consider the linear forms $f_i(x)=x_i-\sum_{i<j}x_j$, in the variables $x_1,\ldots, x_n$. Due to their triangularity property, it is clear that the system of inequalities $f_i(x)\geq 1$ has at least one solution $x_1,\ldots,x_n$ in positive integers. We choose these exponents $x_i$ to inflate $w$. 

Denote by $T_k$ the tableau obtained after Schensted row insertion, into the empty tableau, of $a_1^{x_1}\cdots a_k^{x_k}$. We show by induction that the block condition is satisfied, and that 
each row of $T_k$, when viewed as a word, is an increasing product of letters with exponents equal to $x_i+\sum_{i<j\leq k} \epsilon_j x_j$, 
with $\epsilon_j\in\{-1,0,1\}$, for some $i\leq k$. This is clear for $T_1=a_1^{x_1}$, a tableau with one row. 

Now, insert $a_{k+1}^{x_{k+1}}$ into $T_k$, obtaining $T_{k+1}$. If nothing is bumped, the block condition 
is clearly satisfied, as are the exponent conditions for $T_{k+1}$. Otherwise, some $b^{x_{k+1}}$ is bumped. Moreover, the exponents in the first row are not changed, with the two following exceptions: 1) The exponent of $a_{k+1}$ increases by
$x_{k+1}$. 2) the exponent of $b$ decreases by $x_{k+1}$; note that this is possible (that is, the block condition is satisfied at this row insertion), since its 
exponent before bumping is of the form $x_i+\sum_{i<j\leq k} \epsilon_j x_j$, which is greater than $x_{k+1}$; indeed, this follows from $x_i+\sum_{i<j\leq k}\epsilon_jx_j-x_{k+1}\geq f_i(x)\geq 1$. Now one 
inserts $b^{x_{k+1}}$ in the second row, and so on, and the argument is similar.

Finally, the tableau $T_{n}$, which is $P(w')$, satisfies the required conditions, since one verifies recursively that each step of the previous insertion by factors corresponds to a step of the $N$-insertion of $w$, and that the corresponding rows have the same support.
\end{proof}

\subsection{The mapping $\delta$}\label{delta}

We define a mapping $\delta: A^*\to A^*$ as follows. Define for each subset $B$ of 
$A$, and each letter $x$ in $A$, the element $x^\uparrow_B\in B\cup 1$ to be the smallest letter in $B$ which is greater than $x$, and 
the empty word $1$ if such a letter 
does not exist (that is, if $x\geq \max(B)$). Then we define $\delta(1)=1$, and $\delta(wx)=\delta(w)x^\uparrow_{\Supp(w)}$, for any word $w$ and any letter $x$. 

Concretely, one scans the letters of $w$ from left to right, at each position one searches at the left the smallest letter which is greater than the letter in the current position (it may not exist), and write these letters form left to right. 

Example: let the alphabet be $\{a<b<c<d\}$; then $\delta(acccadbcbac)=ccdcbd$, and the algorithm just described is best seen on a two rows array:
$$
\begin{array}{cccccccccccccccc}
a&c&c&c&a&d&b&c&b&a&c&=&w\\
  &  &  & &c&  &c&d&c&b&d &=&\delta(w)
\end{array}
$$ 

The following lemma is a direct consequence of the definition of the right $N$-algorithm; indeed, the sequence of bumped letters from the first row during the right $N$-algorithm applied to $w$ is precisely the word $\delta(w)$.

\begin{lemma}\label{N-variant} The first row of $N(w)$ is $\Supp(w)$ (viewed as a strictly increasing word) and the remaining $N$-tableau is $N(\delta(w))$.
\end{lemma}

Define, for two subsets $B,C$ of the alphabet, the set $$\D_B(C)=\{c^\uparrow_B \mid c\in C, c^\uparrow_B\neq 1\},$$
which is a subset of $B$. Note that if $B\subseteq C$ and $\min(B)>\min(C)$, then 
\begin{equation}\label{DBCB}
\D_B(C)=B.
\end{equation}

We denote by $\sigma$ the natural bijection associating to each subset of $A$ the increasing product of its elements. Note that if a word $u$ is increasing, then 
\begin{equation}\label{s}
u\equiv_{styl}\sigma(\Supp(u)),
\end{equation}
by Lemma \ref{idemp2}. For later use, we prove the following lemma.

\begin{lemma}\label{deltaN} Let $u_1,\ldots,u_k$ be strictly increasing words such that 
their supports $U_1, ..., U_k$ satisfy $U_1 \supseteq \cdots \supseteq U_k$. Let $x\in A^*$ and $X=\Supp(x)$. Then 
$$
\delta(xu_k\cdots u_1)\equiv_{styl}    \delta(x)\prod_{i=k}^{i=1}\sigma(\D_{U_{i+1}\cup X}(U_{i})),
$$
with the convention that $U_{k+1}=\emptyset$.
\end{lemma}


\begin{proof} It follows directly from the definition of $\delta$ that for any word $w$, $\delta(xw)=\delta(x)\prod_{w=vyv'
}y^\uparrow_{\Supp(xv)}$, where the product is over all factorizations $w=vyv'$, $v,v'\in A^*,y\in A$, and from left to right. Let 
$w=u_k\cdots u_1$; then $\delta(xw)=\delta(x)\prod_{i=k}^{i=1}\prod_{u_i=vyv'}
y^\uparrow_{\Supp(xu_k\cdots 
u_{i+1}v)}$. Note that, in the latter product, the letters in $v$ are less than $y$; hence $y^\uparrow_{\Supp(xu_k\cdots u_{i+1}v)}
=y^\uparrow_{\Supp(xu_k\ldots u_{i+1})}$. Moreover, the supports of the $u_i$ being decreasing from $1$ to $k$ in the inclusion order, we have $
\Supp(xu_k\cdots 
u_{i+1})=\Supp(xu_{i+1})=U_{i+1}\cup X$. 
Thus
$\delta(w)=\delta(x)\prod_{i=k}^{i=1}\prod_{u_i=vyv'}y^\uparrow_{U_{i+1}\cup X}$. 
Finally, note that if a word $m$ is strictly increasing, and $U$ a subset of $A$, then the word $p=\prod_{m=vyv'}y^\uparrow_U$ is 
increasing, so that $p\equiv_{styl}\sigma(\Supp(p))$, by (\ref{s}); thus $p\equiv_{styl}\sigma(\D_U(\Supp(m)))$, since 
$\Supp(p)=\{y^\uparrow_U\mid y\in \Supp(m), y^\uparrow_U\neq 1\}=\D_U(\Supp(m))$. It follows from this that
$\delta(w)\equiv_{styl}\delta(x)\prod_{i=k}^{i=1}\sigma(\D_{U_{i+1}\cup X}(U_i))$.
\end{proof}

\section{A bijection}

\begin{theorem}\label{bijection} The mapping $w\mapsto N(w)$ induces a bijection from the monoid $\Styl(A)$ onto the set of $N$-tableaux on $A$.
\end{theorem}

The theorem is a consequence of several lemmas. 

\begin{lemma}\label{first-column} (i) The tableaux $P(w)$ and $N(w)$ have the same first column, which is $w\cdot \emptyset$.

(ii) Let $\gamma$ be any column. Then $w\cdot\gamma$ is equal to the first column of $N(wu)$, where $u$ is the strictly decreasing word associated to $\gamma$.
\end{lemma}

\begin{proof} (i) We know by Lemma \ref{infl} that $w$ has some inflation $w'$ such that the corresponding rows in $N(w)$ and $P(w')$ have the same support. Hence these two tableaux have the same first column. Moreover $w'\equiv_{styl}w$ by Lemma \ref{idemp2}; thus $w'\cdot \emptyset=w\cdot \emptyset$. Hence $P(w')$ and $P(w)$ have the same first column, by Proposition \ref{first}.

(ii) We know by Proposition \ref{first} that $w\cdot\gamma$ is equal to the first column of $P(wu)$; hence also to the first column of $N(wu)$ by (i).
\end{proof}

Let $\gamma$ be a column on the alphabet $A$. We denote by $\gamma^-$ the column obtained by replacing each letter by the previous one in the alphabet $A$, removing if necessary the smallest letter. The column $\gamma^+$ is defined symmetrically.

\begin{lemma}\label{action-} Let $a=\min(A)$ and  $z=\max(A)$. Let $\gamma$ be a column on $A\setminus z$, and $w\in A^*$ with $A=\Supp(w)$. Then $w\cdot \gamma=a\cup \delta(w)\cdot \gamma^+$ and $(w\cdot \gamma)^-=\delta(w)^-\cdot \gamma$.
\end{lemma}

\begin{proof} By Lemma \ref{first-column}, $w\cdot \gamma$ is the first column of $N(wu)$, where $u$ is the strictly decreasing word having same support as $\gamma$. Since $a$ appears in $w$, $a$ appears in $N(wu)$, necessarily at the bottom of the first column. By Lemma $\ref{N-variant}$, the first column of $N(wu)$ is equal to the first column of $N(\delta(wu))$ with $a$ added at the bottom. 

Now, since $u$ does not involve the letter $z$ and since $w$ involves each letter in $A$, we have $\delta(wu)=\delta(w)u^+$, where $u^+$ is obtained by replacing in $u$ each letter by the next one in the alphabet $A$. Hence the first column of $N(\delta(wu))=N(\delta(w)u^+)$ is by Lemma \ref{first-column} equal to $\delta(w)\cdot \gamma^+$.

It follows from the previous remarks that $w\cdot \gamma=a\cup \delta(w)\cdot \gamma^+$, which implies the lemma.
\end{proof}

\begin{lemma}\label{class} $N(w)$ depends only on the class of $w$ modulo $\equiv_{styl}$.
\end{lemma}

\begin{proof} 
It is enough to show that if $w,w'$ have the same action on the set of columns over $\Supp(w)\cup\Supp(w')$, then $N(w)=N(w')$. Note that by Lemma \ref{support}, the hypothesis implies that they have the same support.

We prove the lemma by induction on $|\Supp(w)\cup \Supp(w')|$; the case where it is empty is clear. Suppose now that $A=\Supp(w)\cup \Supp(w')$ is nonempty and let $a=\min(A)$. 
By hypothesis, $w,w'$ have the same action on $\mathcal C(A)$. 

By Lemma \ref{N-variant}, the first row of $N(w)$, viewed as a set, is $\Supp(w)$, and the remaining tableau is $N(\delta(w))$. Hence the first rows of $N(w)$ and $N(w')$ are equal. 
Note that every letter of $\delta(w)$ and $\delta(w')$ is in the alphabet $A \setminus a$;
hence $\Supp(\delta(w))\cup\Supp(\delta(w'))\subseteq A\setminus a$.

We claim that the action of $\delta(w)$ on $\mathcal C(A\setminus a)$ depends only on the action of $w$ on $\mathcal C(A)$. Indeed, let $\gamma_1$ be a column on $A\setminus a$. Then $\gamma=\gamma_1^-$ is a column on $A\setminus z$, where $z=\max(A)$; note also that $\gamma^+=\gamma_1$, hence by Lemma \ref{action-}, $a\cup \delta(w)\cdot\gamma_1=w\cdot \gamma$, which implies $\delta(w)\cdot \gamma_1=(w\cdot \gamma)\setminus a$.

The claim is also true for $\delta(w')$, so that $\delta(w)$ and $\delta(w')$ have the same action of $\mathcal C(A\setminus a)$. Hence, they have the same action on the set of columns over $\Supp(\delta(w))\cup\Supp(\delta(w'))$. By induction $N(\delta(w))=N(\delta(w'))$. It follows that $N(w)=N(w')$ by Lemma \ref{N-variant}.
\end{proof}

\begin{lemma}\label{NT} Let $T$ be an $N$-tableau. Then $N(RW(T))=T$.
\end{lemma}

\begin{proof}
Let $T$ have $k$ rows, and let $u_1,\ldots,u_k$ be the row-words of the rows from $i=1$ to $i=k$; moreover, let $U_i=\Supp(u_i)$. Then $RW(T)=u_k\cdots u_1$. By Lemma \ref{deltaN}, with $x=1$, we have $\delta(RW(T))=\delta(u_k\cdots u_1)\equiv_{styl}  \prod_{i=k-1}^{i=1}\sigma(\D_{U_{i+1}}(U_{i}))$ (since the factor for $i=k$ is the empty word). Now, by (\ref{DBCB}), one has $\D_{U_{i+1}}(U_{i})=U_{i+1}$, since $U_{i+1}\subseteq U_i$ and $\min(U_i)<\min(U_{i+1})$; therefore $\sigma(\D_{U_{i+1}}(U_{i}))=u_{i+1}$. 
Hence $\delta(RW(T))\equiv_{styl}u_k\ldots u_2=RW(T')$, the row word of the $N$-tableau $T'$ obtained by removing the first 
row from $T$. It follows from Lemma \ref{class} that $N(\delta(RW(T)))=N(RW(T'))$; by induction, this is the 
$N$-tableau $T'$. By Lemma \ref{N-variant}, we deduce that $N(RW(T))$ is equal to $T$, since the support of $RW(T)$ is equal to that of $T$ and therefore to the first row of $T$.
\end{proof}

\begin{proof}[Proof of Theorem \ref{bijection}]
The mapping is well-defined by Lemma \ref{class}. 
Surjectivity follows from Lemma \ref{NT}.

The mapping is injective, since, using Lemma \ref{infl} and its notations, one has $w\equiv_{styl}w'$ by Lemma 
\ref{idemp2}. And $w'\equiv_{plax}RW(P(w'))$ by Section \ref{plactic}, 
and finally $RW(P(w'))\equiv_{styl}RW(N(w))$, by Lemma \ref{idemp2} and Lemma \ref{infl}.
Thus \begin{equation}\label{wequivrNw}
w\equiv_{styl} RW(N(w)),
\end{equation}
which proves injectivity.
\end{proof}

\begin{corollary}\label{right-insert} Let $T$ be an $N$-tableau and $x$ a letter. Then $(T\leftarrow x)=N(RW(T)x)$.
\end{corollary}

\begin{proof} By definition of the $N$-insertion, $N(RW(T)x)=(N(RW(T))\leftarrow x)=(T\leftarrow x)$, by Lemma \ref{NT}.
\end{proof}

\begin{corollary}\label{wdeltaSupp} Let $w\in A^*$. Then $w\equiv_{styl} \delta(w)\sigma(\Supp(w))$.
\end{corollary}

\begin{proof} Let $u_1,\ldots,u_k$ the increasing words corresponding to the rows of $T=N(w)$, from the longest row to the shortest. Then $RW(T)=u_k\cdots u_1$. Let $T'$ obtained from $T$ by removing the first row; then $RW(T')=u_k\cdots u_2$. Moreover, $T'=N(\delta(w))$ by Lemma \ref{N-variant}. By (\ref{wequivrNw}), $w\equiv_{styl}RW(N(w))=u_k\cdots u_1=RW(T') u_1\equiv_{styl} \delta(w) \sigma(\Supp(w))$, by (\ref{wequivrNw}).
\end{proof}

\section{Cardinality and presentation of the stylic monoid}\label{card}

Recall that the {\it Bell number} $B_n$ is the number of partitions of a set with $n$ elements. The first few values, starting with $n=1$, are $1,2,5,15,52,203,877$.

\begin{theorem}\label{main} (i) If the cardinality of $A$ is $n$, then the cardinality of $\Styl(A)$ is $B_{n+1}$.

(ii) $\Styl(A)$ is presented by the plactic relations and the relations $x^2=x$, $x\in A$.
\end{theorem}

We call {\it stylic relations} the plactic relations together with the relations $x^2=x$, $x\in A$. Denote by $\Part(E)$ the set of partitions on a set $E$.

\begin{lemma}\label{bij1} To each $N$-tableau $T$ on $A$, associate the partition $R$ of the set $\Supp(T)$ obtained as follows: 
denoting the rows of $T$ by $R_i$, $i=1,\ldots,k$, from the longest to the shortest, and viewing them as subsets of $A$, 
the parts of $R$ are $R_k,R_{k-1}\setminus R_k,\ldots,R_1\setminus R_2$. This mapping is a bijection
from the set of $N$-tableaux on $A$ onto the set $\bigcup_{B\subseteq A}\Part(B)$. The inverse mapping is defined as follows: 
let $R=\{B_1,\ldots,B_k\}$, ordered in such a way that $\min(B_1)<\ldots<\min(B_k)$; then the rows of the associated 
$N$-tableau, viewed as subset of $A$, are the sets $\bigcup_{i\leq j\leq k}B_j$, $i=1,\ldots,k$.
\end{lemma}

As an illustration, consider the $N$-tableau of Figure \ref{Tableau}, whose rows are $\{a,b,c,d,e\}, \{b,d,e\},\{d,e\}$: then $R=\{\{a,c\}, \{b\}, \{d,e\} \}$. 

\begin{proof} This follows from the bijection between $N$-tableaux and $N$-filtrations, as stated at the beginning of Section \ref{tab}.
\end{proof}

\begin{proof}[Proof of Theorem \ref{main}] (i) The cardinality of $\Styl(A)$ is equal by Theorem \ref{bijection} to the number 
of $N$-tableaux on $A$. This number is by Lemma \ref{bij1} equal to $\sum_{B\subseteq A}|\Part(B)|=\sum_k\binom{n}{k}B_k$, which is well-known to be equal to $B_{n+1}$.

(ii) By Corollary \ref{plax-styl} and Lemma \ref{idemp2}, the stylic relations are satisfied in $\Styl(A)$. 

Conversely, denote by $\equiv$ the congruence of $A^*$ generated by the stylic relations.
Suppose that $u\equiv_{styl}v$;  it is enough to show that $u\equiv v$.
We have by Lemma \ref{class}, $N(u)=N(v)$.
We have $u\equiv u', v\equiv v'$, where $u',v'$ are some inflation of $u,v$ respectively, as 
indicated in Lemma \ref{infl}; by this lemma, and the idempotence of the generators, we have $RW(N(u))\equiv RW(P(u')), 
RW(N(v))\equiv RW(P(v'))$. We have by Section \ref{plactic}, $u'\equiv RW(P(u')), v'\equiv RW(P(v'))$, since $\equiv_{plax}$ 
implies $\equiv$. In conclusion, we have $u\equiv u'\equiv RW(P(u')) \equiv RW(N(u)) = RW(N(v)) \equiv RW(P(v'))
\equiv v' \equiv v$.
\end{proof}

The proof also yields the following corollary.

\begin{corollary}\label{repres} The set of words of the form $RW(T)$, $T$ an $N$-tableau on $A$, is a set of unique representatives of the stylic classes.
\end{corollary}

\begin{corollary} Let $B\subseteq A$. The natural injection $B^*\to A^*$ induces an injection $\Styl(B)\to\Styl(A)$. In other 
words, if two words $u,v$ in $B^*$ have the same action on $\C(B)$, then they have the same action on $\C(A)$.
\end{corollary}

A direct proof of the latter assertion seems not evident.

\begin{proof} This follows since the presentation is support-preserving: if one applies an elementary plactic move, or a move according to $x^2\equiv_{styl} x$, the alphabet does not change. Hence the relations $u\equiv_{styl} v$ in the large alphabet imply the relations in the small alphabet.
\end{proof}

We say that an element $w$ of $\Styl(A)$ is {\it complete} if its support is equal to $A$.

\begin{corollary} If $|A|=n$, then the number of complete elements in $\Styl(A)$ is equal to $B_n$. 
\end{corollary}

\begin{proof} The complete elements correspond in the bijection of Theorem \ref{bijection} to the $N$-tableaux whose support is $A$. Hence their number is $B_n$ by the argument seen in part (i) of the proof of Theorem \ref{main}.
\end{proof}

%
%

\section{Evacuation of partitions}

\subsection{An involution}\label{invol}

Recall that $A$ is a totally ordered finite alphabet. Denote by $\theta$ the unique order-reversing permutation of $A$. It extends uniquely to an anti-automorphism of the free monoid, that we still denote $\theta$. For example, with $A=\{a<b<c<d\}$, $\theta(acdaadc)=baddabd$. The mapping $\theta$ is clearly an involution.

Strictly speaking, $\theta$ depends on $A$ and we denote it $\theta_A$ if necessary. For later use, we note that if $a$ is the smallest element of $A$, and denoting by $i_a:(A\setminus a)^*\to A^*$ the monoid homomorphism sending each letter $x$ in $A\setminus a$ onto the letter that precedes $x$ in the total order of $A$, then
\begin{equation}\label{theta-A-a}
\forall w\in (A\setminus a)^*, \theta_A(w)=i_a\circ\theta_{A\setminus a}(w).
\end{equation}
Both sides are indeed anti-homomorphisms, which coincide on the alphabet $A\setminus a$. Likewise, if $z$ is the largest letter of $A$, and $j_z$ the homomorphism from $(A\setminus z)^*\to A^*$ sending each letter to the next one in the order of $A$, then
\begin{equation}\label{theta-A-z}
\forall w\in (A\setminus z)^*, \theta_A(w)=j_z\circ\theta_{A\setminus z}(w).
\end{equation}

Let us come back to the fixed alphabet $A$ and $\theta=\theta_A$. Clearly, and as is well-known, the plactic relations (see Section \ref{plactic}) are invariant under $\theta$. It follows that $\theta$ induces an anti-automorphism of the plactic 
monoid. Similarly, the stylic relations (see the definition following Theorem \ref{main}) are invariant under $\theta$, and therefore $\theta$ induces an anti-automorphism of 
the stylic monoid. Both anti-automorphisms are involutions, and we denote them with the same notation $\theta$. We thus obtain the commutative diagram of
Figure \ref{commuting}, where the vertical mappings are the canonical quotient homomorphisms. 

\begin{figure}
\begin{center}
\begin{tikzpicture} 
\node at (0, 0) {$A^*$} ;

\node at (0, -2.3) {$\Plax(A)$} ;

\node at (4, -2.3) {$\Plax(A)$} ;

\node at (0, -4.6){$\Styl(A)$};

\node at (4, -4.6){$\Styl(A)$};
    
\node at (4,0) {$A^*$};
  
\node at (2, 0.3) {$\theta$};
  
\node at (2, -2) {$\theta$};
  
\node at (2, -4.2) {$\theta$};
  
\draw [->] (0,-0.3) -- (0,-1.9) ; 
\draw [->] (0,-2.6) -- (0,-4.3) ;
\draw [->] (0.3,0) -- (3.7,0) ;
\draw [->] (0.8,-2.3) -- (3.2,-2.3);
\draw [->] (0.8,-4.6) -- (3.2,-4.6);
\draw [->] (3.9,-0.3) -- (3.9,-2);
\draw [->] (3.9,-2.6) -- (3.9,-4.3);
\end{tikzpicture}
\end{center}
\caption{Commuting homomorphisms and anti-automorphisms}
\label{commuting}
\end{figure}

The plactic monoid is in bijection with Young tableaux. The endomorphism $\theta$ of the plactic monoid is described directly on the set of tableaux by the {\it Schützenberger involution} (\cite[p.127]{Sc}), also called {\it evacuation} (see \cite[3.9]{Sa}, \cite[p.425]{St}). 

We give now a construction on (set-theoretical) partitions, similar to Schützenberger's evacuation, which will be shown to correspond to the involution $\theta$ of the stylic monoid.

Fix the alphabet $A$ and the involution $\theta=\theta_A$. For each nonempty subset $B$ of $A$, we define a mapping $\Delta:\Part(B)\to\Part(B\setminus \min(B))$. 
For this, we order the blocks of each partition on the totally 
ordered set $B$, according to the order of the minimum of the blocks. Therefore, we may speak of the $j$-th block of a partition.

Let $R=\{B_1,B_2,\ldots,B_k\}\in \Part(B)$. Let $x_i=\min(B_i)$; we assume that $x_1<x_2<\cdots<x_k$. 
Let $u_i$ be the strictly increasing word whose support is $B_i$; then $x_i$ is the first letter of $u_i=x_iv_i$.

Consider the word $w=u_1u_2\cdots u_k=x_1v_1x_2v_2\cdots x_kv_k$. We determine an integer $e(R)$ as follows:
\begin{itemize}
\item
Define first $x:=x_1$ and $e:=1$.
\item Look for the smallest letter $y$ at the right of $x$ in $w$: if $y$ is some $x_j$, let $x:=x_j$, $e:=j$ and iterate this step. If $y$ is not an $x_j$, or there is no letter at the right of $x$, then the algorithms stops.
\item Put $e(R):=e$.
\end{itemize}

Let $e=e(R)$. Define $B'_j=(B_j\setminus x_j)
\cup x_{j+1}$ for $j=1,\ldots,e-1$, $B'_e=B_i\setminus x_e$ and $B'_j=B_j$ for $j>e$. 
Then $\Delta(R)$ is the partition whose blocks are the nonempty sets $B'_j$ (only $B'_e$ may be empty, and in this case $e$ must be equal to $k$).

For example, with $A=[8]=\{1,2,3,4,5,6,7,8\}$, and $R=13/28/457/6$  (with evident notations), $w=u_1u_2u_3u_4=(13)(28)(457)(6)$, we have 
$x_1=1,x_2=2,x_3=4, x_4=6$, $e(R)=3$: indeed, at the end of the algorithm, when $x$ is set to $x_3=4$, then $y$ is set to $5$, which is not an $x_j$. Thus $\Delta(R)=23/48/57/6$, which is a partition of the set $[8]\setminus 1=\{2,3,4,5,6,7\}$. This example is also given in another form in Figures \ref{partition}, \ref{Delta1} and \ref{Delta2}.


For each subset $B$ of $A$, the {\it evacuation mapping} $\evac$, from $\Part(B)$ into itself, is then recursively defined as follows. If $B$ is empty and $R\in \Part(B)$, then $\evac(R)=R$ ($R$ is here the empty partition). Otherwise, let $R\in\Part(B)$, $B\subseteq A$, $B$ nonempty. Let $b=\min(B)=\min(R)$. Then, with the notation $e(R)$ above, $\evac(R)$ is the partition on $B$, obtained from $\evac(\Delta(R))$ by adding $\theta_A(b)$ to its $e(R)$-th block (and creating this block if necessary; note that it is then the last block).

Note that the definition of evacuation implies that $\theta(b)$ is the largest letter in $R$, and that
\begin{equation}\label{evac-z}
\evac(\Delta(R))=\evac(R)\setminus \theta(b).
\end{equation}

Denote by $\pi$ the mapping associating to each word $w$ the partition corresponding bijectively to the $N$-tableau $N(w)$, as 
described in Lemma \ref{bij1}; see the example following it. 

\begin{theorem}\label{evac} One has $\pi(\theta(w))=\evac(\pi(w))$ for any word $w$.
\end{theorem}

In other words, the involutive anti-automorphism $\theta$ of the stylic monoid corresponds at the level of partitions to evacuation of partitions. 
We shall prove the theorem in Section \ref{proof-evac}, after a detour through a generalization of jeu de taquin, which is interesting for itself.

For later use, we note that evacuation, as defined above, depends on the mapping $\theta$, which depends in turn on $A$, and is 
therefore denoted $\evac_A$ if necessary. As for $\theta$, we have the following rules. We use the functions $i_a$ and $j_z$ defined before and after (\ref{theta-A-a}), naturally extended to partitions.

\begin{lemma} Let $a$ (resp. $z$) be the smallest (resp. largest) letter of $A$. 
\begin{equation}\label{evac-A-a}
\forall R\in \Part(B), B\subseteq A\setminus a, \evac_A(R)=i_a\circ\evac_{A\setminus a}(R).
\end{equation}
\begin{equation}\label{evac-A-z}
\forall R\in \Part(B), B\subseteq A\setminus z, \evac_A(R)=j_z\circ\evac_{A\setminus z}(R).
\end{equation}
\end{lemma}

\begin{proof} Note that  the function $\Delta$ is independent of the alphabet. Let $R\in \Part(B)$, $B\subseteq A$, $b=\min(B)=\min(R)$, and let $R'=\Delta(R)$, $e=e(R)$, $R_1=\evac_A(R)$. Then, by definition of evacuation, $R_1$ is obtained from $\evac_A(R')$ by inserting $\theta_A(b)$ into its $e$-th block.

Suppose that  $B\subseteq A\setminus a$. By definition, $\evac_{A\setminus a}(R)$ is obtained from $\evac_{A\setminus a}(R')$ by inserting $\theta_{A\setminus a}(b)$ into its $e$-th block. We have clearly $R'\in \Part(B'), B'\subseteq A\setminus a$; hence by induction, $\evac_A(R')=i_a\circ\evac_{A\setminus a}(R')$; now, since by (\ref{theta-A-a}), $\theta_A(b)=i_a\circ\theta_{A\setminus a}(b)$, inserting $\theta_A(b)$ into the $e$-th block of $\evac_A(R')$ amounts to first inserting $\theta_{A\setminus a}(b)$ into the $e$-th block of
$\evac_{A\setminus a}(R')$ and then applying $i_a$. This proves (\ref{evac-A-a}), and (\ref{evac-A-z}) is proved similarly.
\end{proof}

\subsection{Skew-partitions with a hole}\label{skew-hole}
Comparison of the definitions below with Ferrers diagram, lower poset ideals in $\N^2$, Young tableaux, skew Young tableaux, and paths in Young's lattice may be useful (see \cite{Sa, St}), since what we do now is very similar, after a change of the order on $\N^2$.

Let $\p=\N\setminus 0$. We consider the order on $\p^2$, denoted $
\preceq$, such that the covering relations are $(1,y)\preceq (1,y+1)$ and $(x,y)\preceq (x+1,y)$; its Hasse diagram is 
represented in Figure \ref{order-preceq}, where one increases in the order by going up or to the right (north or east). When we speak of the order on $\p^2$, it will be always the order $\preceq$.

A {\it lower ideal} in a poset $E$ is a subset $I\subseteq E$ such that for 
any elements $\alpha\leq \beta$ in $E$, if $\beta\in I$, then $\alpha\in I$. 

We call {\it $A$-labelling} of a finite poset a bijective increasing mapping from the poset into the totally ordered set $A$. The mapping is indicated by labelling the vertices of the Hasse diagram of the poset. 

It is easy to see that 
a finite lower ideal in $\p^2$ (with the order $\preceq$) corresponds bijectively to a composition: the parts of the composition are the number of points in the ideal with equal $y$-coordinate, starting from the bottom ($y=1$); see Figure \ref{compos} for an example, with the composition $(2,2,3,1)$. 

The order induced on compositions\footnote{Another order on compositions, with more covering relations, has been considered in \cite{BBMD}.} by the inclusion of finite lower ideals of $\p^2$ is easily described by its covering relation $\to$: $C\to C'$ if and only if either $C'$ is obtained by increasing one part of $C$ by 1, or if $C'$ is obtained by adding the new part 1 at the end of $C$ (so that the number of covering compositions of $C$ is one more than the number of parts of $C$). For example, $(2,2,3,1)\to(2,3,3,1)$ and $(2,2,3,1)\to(2,2,3,1,1)$. 

Note that the set of finite lower ideals of  $\p^2$, denoted $\mathcal I$, is a lattice for the inclusion order. For simplicity, we say {\it ideal} instead of ``finite lower ideal of $\p^2$".

\begin{figure}\centering
\begin{minipage}[b]{0.33\linewidth}\centering
$\begin{array}{ccccccccc}
(1,4)&-&(2,4)&-&(3,4)&\cdots\\
|\\
(1,3)&-&(2,3)&-&(3,3)&\cdots\\
| \\
(1,2)&-&(2,2)&-&(3,2)&\cdots\\
| \\
(1,1)&-&(2,1)&-&(3,1)&\cdots
\end{array}$
\caption{
}
\label{order-preceq}
\end{minipage}
\hfill
\begin{minipage}[b]{0.33\linewidth}\centering
$\begin{array}{ccccc}
\circ\\
|\\
\circ&-&\circ&-&\circ\\
| \\
\circ&-&\circ\\
| \\
\circ&-&\circ
\end{array}$
\caption{
}\label{compos}
\end{minipage}
\end{figure}

Let $I$ be an ideal of $(\p,\preceq)$, or equivalently, a composition. Consider an {\it $A$-labelling} of 
$I$, with $I$ considered as a poset with the order $\preceq$. To such a labelling $I\to A$, we associate the partition $\{B_1,\ldots,B_k\}$ of $A$, where $B_i$ is 
the set of labels of the points in $I$ with $y$-coordinate $i$; in the example of Figure \ref{partition}, one has $k=4$ and 
$B_1=\{1,3\},B_2=\{2,8\},B_3=\{4,5,7\},B_4=\{6\}$. Observe that one has necessarily $
\min(B_1)<\min(B_2)<\cdots<\min(B_k)$ since the labelling is increasing. 

Note that a (set-theoretical) partition on a finite totally ordered set $A$ may be uniquely represented by the sequence of its 
blocks $
(B_1,\ldots,B_k)$ with $\min(B_1)<\min(B_2)<\cdots<\min(B_k)$. It follows that increasing $A$-labellings of
ideals of $\p^2$, of cardinality $|A|$, correspond bijectively to partitions of $A$. We call $I$ the {\it shape} of the partition, if the latter corresponds to an $A$-labelling of $I$.

Call {\it path} in a poset  a sequence of elements such that each element covers the previous one. Note that an $A$-labelling (hence a partition) is equivalent to a path in $\mathcal I$, starting 
from the singleton $\{(1,1)\}$; equivalently, to a path of compositions $C_1\to\cdots\to C_k$ with $C_1=(1)$; for 
example, in Figure \ref{partition}, it is the sequence 
$(1)\to(1,1)\to(2,1)\to(2,1,1)\to(2,1,2)\to(2,1,2,1)\to(2,1,3,1)\to(2,2,3,1)$. 

Given two ideals $I,J$ in $\mathcal I$, the set $I\setminus J$ will be called a {\it skew ideal}; clearly, one may assume that 
$J\subseteq I$, what we assume in the sequel. If $S=I\setminus J$ is a skew ideal, then a point $H\in S$ such that $S\setminus H$ is still a skew ideal is called a {\it corner} of $S$. We call it a {\it lower corner} if $J\cup H$ is an ideal, and an {\it upper corner} if $I\setminus H$ is an ideal. For example, in Figure \ref{skew}, with $S$ the set of labelled points, the lower corners are $(1,3),(2,2), (3,1)$ and the upper corners are $(1,4),(3,3),(4,1)$.

A {\it skew partition } is an $A$-labelling of a skew ideal; the latter is called its {\it shape}. Equivalently, a skew partition is an upwards path in the 
Hasse diagram of $\mathcal I$; equivalently, an upwards path in the Hasse diagram of compositions with the order $\preceq$. See Figure \ref{skew}, where the 
sequence of compositions is $(2,1)\to(2,1,1)\to(3,1,1)\to(3,1,2)\to (3,2,2)\to(3,3,2)\to(4,3,2)\to (4,3,2,1)\to(4,3,3,1)$.

We call {\it pointed skew ideal} a pair $(S,H)$ of a skew ideal $S$, together with some point $H\in S$. 

Finally, we call {\it skew partition with a hole} an $A$-labelling of 
subset $S\setminus H$, where $(S,H)$ is a pointed skew ideal. We call $S$ the {\it shape} and $H$ the {\it hole}; note that the hole has no label. We call the hole {\it upper} (resp. {\it lower}) if $H$ is an upper (resp. lower) corner of $S$; otherwise, the hole is {\it inner}.
For example, in 
Figure \ref{hole}, the hole is the point of coordinates $H=(1,3)$, indicated by a $\circ$, and is inner.

\begin{figure}
\begin{minipage}{0.33\linewidth}\centering
$\begin{array}{ccccc}
6\\
|\\
4&-&5&-&7\\
| \\
2&-&8\\
| \\
1&-&3
\end{array}$
\caption{}\label{partition}
\end{minipage}
\hfill
\begin{minipage}{0.33\linewidth}\centering
$\begin{array}{ccccccc}
7\\
|\\
1&-&3&-&8\\
| \\
*&-&4&-&5\\
| \\
*&-&*&-&2&-&6
\end{array}$
\caption{}\label{skew}
\end{minipage}
\end{figure}

\begin{figure}
\begin{minipage}{0.33\linewidth}\centering
$\begin{array}{ccccccc}
7\\
|\\
\circ&-&3&-&8\\
| \\
1&-&4&-&5\\
| \\
*&-&*&-&2&-&6
\end{array}$
\caption{}\label{hole}
\end{minipage}
\hfill
\begin{minipage}{0.33\linewidth}\centering
$\begin{array}{ccccccc}
7\\
|\\
3&-&\circ&-&8\\
| \\
1&-&4&-&5\\
| \\
*&-&*&-&2&-&6
\end{array}$
\caption{}\label{down}
\end{minipage}
\end{figure}

%

\subsection{Jeu de taquin on skew partitions}

Given a skew partition with a hole $S$, we define two types of moves, which change it into another skew partition, with or without hole.

The {\it downward move} is defined as follows. If $H$
is an upper hole, one removes it and one obtains a skew partition (without hole).
If $H$ is not an upper hole, then there may be one or two points in $S$ covering $H$. In the first case, the point $K$ covering $H$ becomes the new hole, and $H$ gets the label previously on $K$. In the second case, let $K,L$ be the two points, with respective labels $x,y$ and suppose that $x<y$ in $A$; then $K$ becomes the new hole, and $x$ becomes the new label of $H$. One obtains a new skew partition with a hole. For example, the downward move applied in Figure \ref{hole} gives Figure \ref{down}. Observe that the hole of the new skew diagram is further from the minimum $(1,1)$ in the Hasse diagram of $\p^2$.

The {\it upward move} is defined similarly by looking to the point covered by $H$. Note that $H$ can cover at most one point.

A {\it downward slide} on a skew partition $R$ is defined as follows; let $I\setminus J$ be its shape. If $J$ is empty, then $R$ is a partition and the slide is completed, producing $R$.
If $J$ is nonempty, choose a point $H$ that is a maximal element in $J$. Then $(H\cup (I\setminus J),H)$ is a pointed skew diagram, with lower corner $H$, and 
$R$ together with $H$ is a skew partition with the hole $H$. We then apply iteratively downward moves, until one obtains a skew partition without hole (and we call $H'$ the hole that was removed in the last step).
The fact that this ends in finitely many steps follows from the observation above about distance
from $(1,1)$.
Observe that the new skew diagram is of the form $I'\setminus J'$, where $J'=J\setminus H$ and $I'=I\setminus H'$. Each downward slide is determined on the initial skew partition $R$ by a {\it trail}, which is the set of labels obtained starting form $H$ and choosing iteratively the smallest label among the covering points; see Figure \ref{taquin1}, where the trail is indicated by bold numbers; the slide is then obtained by sliding downward (in the poset) the labels in the trail, see Figure \ref{taquin2}.

Finally, {\it downward jeu de taquin} on a skew partition is applying to it iteratively a sequence of downward slides until a partition is obtained. Note that there are several ways to do it, since one has to choose a point $H$ for each slide, and there may be several choices. The final partition is however unique, as stated below.

An example is given in Figures \ref{taquin1} - \ref{taquin4}. Each slide is indicated by its trail in bold.

\begin{figure}
\begin{minipage}[b]{0.33\linewidth}\centering
$\begin{array}{ccccccc}
7\\
|\\
\bf 1&-&\bf 3&-&\bf 8\\
| \\
\circ&-&4&-&5\\
| \\
*&-&*&-&2&-&6
\end{array}$
\caption{}\label{taquin1}
\end{minipage}
\hfill
\begin{minipage}[b]{0.33\linewidth}\centering
$\begin{array}{ccccccc}
7\\
|\\
 3&-& 8\\
| \\
1&-&4&-&5\\
| \\
*&-&\circ&-&\bf 2&-&\bf 6
\end{array}$
\caption{}\label{taquin2}
\end{minipage}
\end{figure}

\begin{figure}
\begin{minipage}[b]{0.33\linewidth}\centering
$\begin{array}{ccccccc}
\bf 7\\
|\\
\bf 3&-& 8\\
| \\
\bf 1&-&4&-&5\\
| \\
\circ&-&2&-&6
\end{array}$
\caption{}\label{taquin3}
\end{minipage}
\hfill
\begin{minipage}[b]{0.33\linewidth}\centering
$\begin{array}{ccccccc}

 7&-& 8\\
| \\
3&-&4&-&5\\
| \\
1&-&2&-&6
\end{array}$
\caption{}\label{taquin4}
\end{minipage}
\end{figure}

Define for each word $w$ the increasing rearrangement $\overline w$ of $w$; for example, $\overline{bacbdbc}=abbbccd$.
For each skew partition $R$, with or without hole, we define its {\it row-word} $RW(R)$ as follows: suppose that the 
shape of $R$ is the skew ideal $I\setminus J$, where the largest $y$-coordinate of a point in $I$ is $k$; denote by $u_i$ 
the word obtained by reading from left to right the labels in $R$ located in the line of $y$-coordinate $i$. Then 
$$
RW(R)=u_k\overline{u_{k-1}u_k}\cdots\overline{u_1\cdots u_{k-1}u_k}.
$$
Note that $u_k$ is already increasing, since the labelling is. For example, for the skew partition in Figure \ref{skew}, its row-word is $7 \,1378 \,134578\, 12345678$, while the row-word of 
the skew partition with hole of Figure \ref{hole} is $7\, 378\, 134578 \,12345678$. Observe that this definition is such that for a 
partition $R$, corresponding to the $N$-tableau $T$, one has $RW(R)=RW(T)$, as follows from Lemma \ref{bij1}.

\begin{theorem}\label{normal} The partition obtained by downward jeu de taquin from a skew partition is independent of 
the choices of the lower corners during the algorithm.
\end{theorem}

\begin{lemma}\label{aua} Let $a\in A, u\in A^*$ be such that each letter in $u$ is greater or equal to $a$. Then $aua\equiv _{styl}ua$.
\end{lemma}

\begin{proof}
It is enough to show that for each column $\gamma$, $(aua)\cdot\gamma=(ua)\cdot \gamma$. This is equivalent to the 
fact that $a$ fixes $(ua)\cdot\gamma$, which will follow, by Lemma \ref{idemp} (i), from the fact that $a$ appears in the column $(ua)
\cdot\gamma$. Now, $a$ appears in $a\cdot\gamma$; and, since the letters in $u$ are all greater or equal to $a$, using 
recursively Lemma \ref{idemp} (i) and (ii), we obtain that $a$ appears in $(ua)\cdot\gamma$.
\end{proof}

\begin{proof}[Proof of Theorem \ref{normal}] Let $R$ be a skew partition and $R_0$ a partition obtained by downward jeu 
de taquin applied to $R$, for some choices of the lower corners. We claim that 
\begin{equation}\label{RequivR0}
RW(R)\equiv_{styl} RW(R_0).
\end{equation}
The claim being admitted, suppose 
that we obtain another partition $R_1$ by downward jeu de taquin; by the claim, we have $RW(R)\equiv_{styl} RW(R_1)$. Let $T_i$ be the $N$-tableau corresponding to the 
partition $R_i$ through the natural bijection of Lemma \ref{bij1}. Then by the observation before the theorem, we have 
$RW(T_0)\equiv_{styl}RW(R_0)\equiv_{styl}RW(R)\equiv_{styl}RW(R_1)\equiv_{styl}RW(T_1)$. Thus $T_0=T_1$ by Corollary
\ref{repres}, and finally $R_0=R_1$ by Lemma \ref{bij1}.

We prove now the claim. It is enough to prove that the stylic class of the row-word is invariant under downward moves of skew 
partitions with holes. Thus let $R'\to R''$ be such a move. The two cases two consider are: (i) shifting the hole to the right; 
(ii) shifting the hole above.

In case (i), the row word does not change. In case (ii), let $i$ and $i+1$ the indices of the rows where the move occurs; note that the hole in $R'$ is then in the first column ($x$-coordinate 1) and in row $i$. 
Denote by $u_j$ the row-word of row $j$ of $R'$. Then $u_{i+1}=av$, with $a$ smaller than each letter in 
$v$, $u_{i}$, $u_{i+2}, u_{i+3}$, \ldots. The row-word of the $i$-th and $i+1$-th rows of $R''$ are $au_i$ and $v$ respectively. For $j\neq i,i+1$, the rows of $R'$ and $R''$ are identical.
Let $k$ be the number of rows in $R'$ and $R''$ (row $k$ of $R''$ may be empty, when $i+1=k$, but this does not change the argument that follows).

For some words $x,y$, $$RW(R')=x(\overline{u_{i+1}u_{i+2}\cdots u_k})(\overline {u_iu_{i+1}\cdots 
u_k})y,$$ and 
$$RW(R'')=x(\overline{vu_{i+2}\cdots u_k})(\overline{au_ivu_{i+2}\cdots u_k})y. $$

Thus it is enough to show that $$(\overline{u_{i+1}u_{i+2}\cdots u_k})(\overline {u_iu_{i+1}\cdots 
u_k})\equiv_{styl}(\overline{vu_{i+2}\cdots u_k})(\overline{au_ivu_{i+2}\cdots u_k}).$$
But the left word is $$(\overline{avu_{i+2}\cdots u_k})(\overline {u_iavu_{i+2}\cdots u_k})=a(\overline{vu_{i+2}\cdots u_k})a(\overline {u_ivu_{i+2}\cdots u_k})$$
and the right word is $(\overline{vu_{i+2}\cdots u_k})a(\overline{u_ivu_{i+2}\cdots u_k})$. Thus the congruence follows from Lemma 
\ref{aua}.
\end{proof}


\subsection{Properties of the mappings $\Delta$ and $\pi$.}\label{deltaPi}

The operator $\Delta$ of Section \ref{invol} may be computed as follows: let $R$ be a partition of a subset of $A$, viewed as in 
Section \ref{skew-hole} as an $A$-labelling of an ideal in $\p^2$. Note that $a=\min(R)$ is in position $(1,1)$; remove it from the 
labels, obtaining a skew partition $R\setminus a$. Then $\Delta(R)$ is the partition obtained by downward jeu de taquin on 
$R\setminus a$. See Figures \ref{Delta1} and \ref{Delta2} for an example, which is the same as the one illustrating the 
definition of $\Delta$ in Section \ref{invol}.

\begin{figure}\centering
\begin{minipage}[b]{0.33\linewidth}\centering
$\begin{array}{ccccc}
6\\
|\\
\bf 4&-&\bf 5&-&\bf7\\
| \\
\bf2&-&8\\
| \\
\circ&-&3
\end{array}$
\caption{}\label{Delta1}
\end{minipage}
\hfill
\begin{minipage}[b]{0.33\linewidth}\centering
$\begin{array}{ccccc}
6\\
|\\
5&-&7\\
| \\
4&-&8\\
| \\
2&-&3
\end{array}$
\caption{}\label{Delta2}
\end{minipage}
\end{figure}

Consider a nonempty word $w\in A^*$, and let $x$ denote a letter appearing in $w$. Denote by $w\setminus x$ the word obtained by removing all $x$'s from $w$. 

Recall that if two words are equal modulo $\equiv_{styl}$, then they have the same underlying alphabet, and in particular the same smallest letter. The next result shows the compatibility of the operations of removing the smallest letter, and the link with $\Delta$. Recall that for any word $w$, $\pi(w)$ is the partition bijectively associated to the $N$-tableau $N(w)$.

\begin{lemma}\label{removing} (i) If $u\equiv_{styl}v$, with smallest letter $a$, then $u\setminus a\equiv_{styl}v\setminus a$; in particular, $\pi(u\setminus a)=\pi(v\setminus a)$. The same holds when removing the largest letter.

(ii) If $a$ is the smallest letter in $w$, then $\pi(w\setminus a)=\Delta(\pi(w))$.

(iii) If $z$ is the largest letter of $w$, then $\pi(w\setminus z)=\pi(w)\setminus z$.
\end{lemma}

\begin{proof} (i) The stylic congruence is generated by the plactic relations and the idempotence relations. Therefore, it 
suffices to prove the statement when $u,v$ differ by an elementary step of this congruence, and we may assume that this 
step involves an $a$. If it is a plactic step, then since $a$ is the smallest letter, the step amounts to replace $aba$ (resp. 
$bab$, resp. $acb$, resp. $bac$) by $baa$ (resp. $bba$, resp. $cab$, resp. $bca$) in one of the words $u,v$, obtaining the other (we have $a<b<c$); this step becomes the identity when the $a$'s are 
removed. If the step is replacing $aa$ by $a$, or conversely, then it becomes the identity too, when the $a$'s are removed.

The second assertion follows from the bijection $\pi$ between the stylic monoid and the set of partitions of subsets of $A$. The last one by symmetry.

(ii) We have by (\ref{wequivrNw}) and the definition of the mapping $RW$ on partitions, $w\equiv_{styl} RW(N(w))
=
RW(\pi(w))$. By (i) we have $w\setminus a\equiv_{styl} RW(\pi(w))\setminus a$. Now, by the definition of $RW$, we have $RW(\pi(w))
\setminus a=RW(\pi(w)\setminus a)$; here $\pi(w)\setminus a$ denotes the skew partition, obtained by removing $a$ from the partition $\pi(w)$. We now apply downward jeu 
de taquin to $\pi(w)\setminus a$, obtaining the partition $R_0$; the latter is by what we have seen above equal to $\Delta(\pi(w))$. By 
(\ref{RequivR0}), $RW(\pi(w)\setminus a)\equiv_{styl} RW(R_0)$. Thus finally, $w\setminus a\equiv_{styl}RW(\Delta(\pi(w)))$, 
and therefore $\pi(w\setminus a)=\pi(RW(\Delta(\pi(w))))=\Delta(\pi(w))$, the last equality by Lemma \ref{NT}.

(iii) We claim that $\delta(w\setminus z)=\delta(w)\setminus z$. The claim being admitted, (iii) follows by induction from 
Lemma \ref{N-variant}.

We prove the claim by induction on $|w|$. 
If $w$ is empty it is clear. So we may assume that ($\ast$) $\delta(w\setminus z)=\delta(w)\setminus z$ and we prove it for 
$wx$, $x$ 
being some letter. We have $\delta((wx)\setminus z)=\delta((w\setminus z)(x\setminus z))=\delta(w \setminus z)t$, where 
$t=1$ if $x=z$, and $t=x^\uparrow_{\Supp(w \setminus z)}$ if $x<z$. On the other hand, $
\delta(wx)=\delta(w)x^\uparrow_{\Supp(w)}$, hence $\delta(wx)\setminus z=(\delta(w)\setminus z)(x^\uparrow_{\Supp(w)}
\setminus z)$. By ($\ast$), it is therefore enough to show that $t=x^\uparrow_{\Supp(w)}\setminus z$. If $x=z$, both sides are 
equal to $1$, since $z$ is the maximum letter. Suppose now that $x<z$. We have to show that ($\ast\ast$) $x^\uparrow_{\Supp(w 
\setminus z)}=x^\uparrow_{\Supp(w)}\setminus z$.
If there exists an element $y$ in $\Supp(w)$ such that $x<y<z$, then, taking $y$ minimum, both sides of ($\ast\ast$)  are equal to 
$y$; if no such $y$ exists, then both sides are equal to $1$, because $x^\uparrow_{\Supp(w)}=z$ or $1$.
\end{proof}

\subsection{Growth diagram}

Recall that a partition on $A$ is equivalent to a path in the Hasse diagram of the poset of compositions, see Section \ref{skew-hole}. Given a partition $R$ on 
$A$, consider the sequence of partitions $R$, $\Delta(R)$, $\Delta^2(R)$,\ldots,$\Delta^{n}(R)$, with $n=|A|$; note that these partitions are on different sets. 
Draw from left to right the $n$ paths of compositions associated with these partitions on a pyramid, each path being represented diagonally upwards, direction 
north-east; see Figure \ref{growth-diagram}, looking only at the north-east arrows $\nearrow$, and disregarding the north-west arrows $\nwarrow$. For example, 
the path $1\to 11\to12\to121\to221\to222$ is associated with the partition $R=15/23/46$, and the path $1\to 11 \to 111 \to 211 \to 212$ is associated with the partition 
$25/3/46=\Delta(R)$. 

We complete this diagram by adding north-west arrows $\nwarrow$, see the figure; at this point it is not clear that these arrows are also covering relations, but it will be proved soon. We call this the {\it evacuation pyramid of $R$}. It follows from the definition of the evacuation that the right side of the pyramid (which goes north-west) represents the path of compositions associated to $\evac(R)$.

Note that the pyramid is formed of rhombuses, that we describe in a moment (the situation, following the work of Sergey Fomin, is quite similar to the one of standard Young tableaux and partitions of integers, see \cite[Proposition A1.2.7]{S}).

\begin{figure}[h!]
\centering
\includegraphics[scale=0.8]{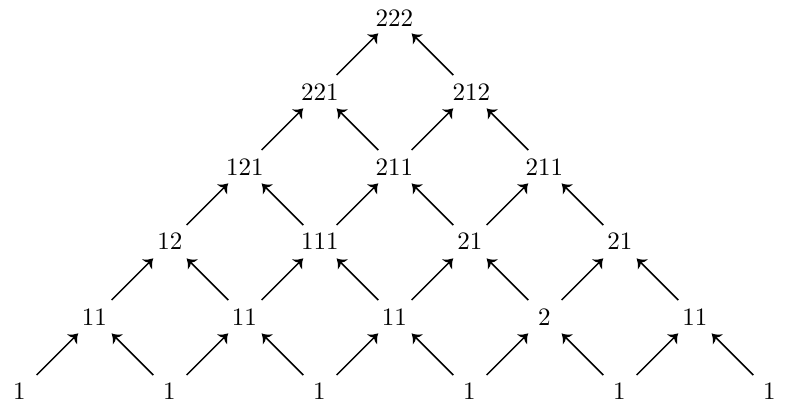} 
\caption{Growth diagram: evacuation of partition $15/23/46$}\label{growth-diagram}
\end{figure}


Before that, we describe the intervals of length 2 in the poset of compositions. By inspecting the definition of the covering relation in this poset, one sees that such an interval is always of cardinality 3 or 4; that is, if $C_1\to C_2\to C_3$, then either $C_2$ is unique and we let $C'_2=C_2$, or there is another composition $C'_2$ such that $C_1\to C'_2\to C_3$. We take this notation below.

\begin{proposition} Each arrow in the evacuation pyramid of $R$ is a covering relation of the poset of compositions. The pyramid may be recursively constructed, starting from the 
bottom row and the leftmost path by applying the following rule: if the two leftmost arrows $C_1\to C_2\to C_3$ of a rhombus are known, then the missing composition of the 
rhombus is $C'_2$.\end{proposition}

\begin{proof} We claim that: for each rhombus in the pyramid: (i) all its sides are covering relations; and (ii) if its leftmost arrows are $C_1\to C_2\to C_3$, then the fourth 
composition is $C'_2$.

To prove the claim, by construction of the pyramid, it is enough to prove it for a rhombus located on the two leftmost north-east paths. Also, since the pyramid obtained by removing the largest element of $A$ is obtained by removing the rightmost north-east sequence, it is enough to prove the claim for the upper rhombus in the pyramid.

Denote by $C_1,C_2,C_3,C_4$ the compositions in this rhombus, as indicated in Figure \ref{rhombus}. 

\begin{figure}
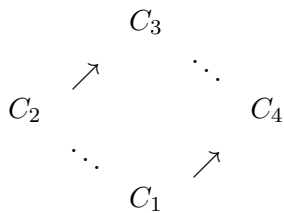

$\begin{array}{ccccc}
&&C_3&&\\
&\nearrow&&\ddots&\\
C_2&&&&C_4\\
&\ddots&&\nearrow&\\
&&C_1&&
\end{array}$
\caption{The upper rhombus}\label{rhombus}
\end{figure}

Let $z$ be the largest letter in $A$ and $x$ be the letter in $A$ which is the last label in the trail determined by the 
downward slide in the computation of $\Delta(R)$ ($x=7$ in Figure \ref{Delta1}). Observe that: $C_3$ is the shape of $R$; 
$C_2$ is the shape of $R\setminus z$; $C_4$ is the shape of $\Delta(R)$, that is, of $R\setminus x$; $C_1$ is the shape 
of $\Delta(R)\setminus z$.

Suppose first that $x\neq z$. Then $x,z$ lie in different parts $P$ and $Q$ (respectively) of $C_3$ (identifying parts of a composition with a 
horizontal subset with given $y$-coordinate of an ideal, as in Figure \ref{compos}) which have respectively size $a,b$. It follows that: $C_2$ is $C_3$ with size of $Q$ replaced by $b-1$; $C_4$ is $C_3$ with size of $P$ replaced by $a-1$; $C_1$ is $C_4$ with size of $Q$ replaced by $b-1$, hence $C_1$ is also $C_2$ with size of $P$ replaced by $a-1$.

Note that if $a=1$ (resp. $b=1$) then $P$ (resp. $Q$) is the last part of the composition (i.e. the top part of the shape) thus if $x\neq z$ one cannot have $a=b=1$. Also, if $a=1$ (resp. $b=1$), $a-1=0$ (resp. $b-1=0$) means that we removed that last part.

Therefore, the interval $[C_1,C_3]$ is equal to $\{C_1,C_2,C_4,C_3\}$ of cardinality 4, which proves the claim in this case. 


Suppose now that $x=z$. Then $z$ lies in a part $P$ of $C_3$, that has size $b$, and it follows that: $C_2$ is $C_3$ with size of $P$ replaced by $b-1$, as is $C_4$; $C_1$ is $C_4$ with size of $P$ replaced by $b-2$. In the particular cases where $b=2$ or $b=1$, $P$ is the last part of the composition. In the first case, for $b-2=0$, one has to remove the last part. In the second case, $b=1$, then the size of the part below is also 1 and one has to remove $P$ to obtain $C_2$ and $C_4$ and remove both parts to obtain $C_1$. Thus the interval $[C_1,C_3]$ is equal to $\{C_1,C_2,C_3\}$ of cardinality 3, which proves the claim in this case.

\end{proof}

\begin{corollary}\label{Deltaev} Let $R$ be a partition on $A$ and $z$ be the largest letter in $A$. Then $\evac(R\setminus z)=\Delta(\evac(R))$.
\end{corollary}


\begin{proof} By the previous proposition, the construction of the evacuation pyramid of $R$ is left-right symmetric. Thus, since the first north-west path from the right (the right side of the pyramid) represents $\evac(R)$, the second one represents the partition $\Delta(\evac(R))$. 
But the evacuation pyramid of $R\setminus z$ is obtained by removing form the whole pyramid its right side. The equality follows.
\end{proof}

\subsection{Proof of the evacuation theorem}\label{proof-evac}

\begin{lemma}\label{R1R2} Let $R_1,R_2$ be two partitions on $A$, with largest element $z$. Suppose that $R_1,R_2$ have the 
same number of blocks, that $R_1\setminus z=R_2\setminus z$ and  that $\Delta(R_1)=\Delta(R_2)$. Then $R_1=R_2$.
\end{lemma}

\begin{proof} Suppose that $R_1\neq R_2$. Then, identifying partitions and labelled ideals in $\p^2$, $z$ has $y$-coordinate $y_i$ in $R_i$ ($i=1,2$), and $y_1\neq y_2$. We have $\Delta(R_1)=\Delta(R_2)$. Thus for at least one of the partitions $R_i$, the $y$-coordinate of $z$ in $R_i$ and in $\Delta(R_i)$ must differ, and we may assume that $i=1$; then the trail corresponding to 
the computation of $\Delta(R_1)$ is the first column of $R_1$, ending at $z$, and the corresponding row of $R_1$ contains only $z$. Thus $\Delta(R_1)$ has 
one row 
less than $R_1$ and $z$ is in the upper row, and first column, of $\Delta(R_1)=\Delta(R_2)$. 
Since $R_1,R_2$ have the same number of rows, $\Delta(R_2)$ has 
one row less than $R_2$, too; this is possible if and only if the trail in $R_2$ is the first column and $z$ is at the top; then $y_1=y_2$, a contradiction.
\end{proof}

\begin{proof}[Proof of Theorem \ref{evac}] The proof is by double induction on $|A|$ and $|w|$.
The theorem is clear if $A$ is empty. Now let 
$A$ be nonempty, with $a,z$ respectively the smallest and largest element. 
Let $w\in A^*$. 

1. We suppose first that $a,z$ appear in $w$.
By induction on the length of $w$, we 
have $\pi(\theta(w\setminus a))=\evac(\pi(w\setminus a))$ and $
\pi(\theta(w\setminus z))=\evac(\pi(w\setminus z)).$

Let $R_1=\pi(\theta(w))$ and $R_2=\evac(\pi(w))$. We have to show that $R_1=R_2$ and do it by verifying the hypothesis of Lemma \ref{R1R2}. 

First, note that, for any word $u$, the number of blocks of $\pi(u)$ is equal to the length of the first column of $N(u)$, hence to the length of the first column of $P(u)$, by Lemma \ref{first-column} (i). This is by Schensted's theorem equal to 
the length of the longest strictly decreasing subword of $u$. Now, the lengths of the longest strictly decreasing subword of $w$ 
and of $\theta(w)$ are clearly equal. It follows that $\pi(w)$ and $\pi(\theta(w))$ have the same number of blocks. 
Moreover the shape of $\evac(\pi(w))$ is equal to that of $\pi(w)$. Hence $R_1$ and $R_2$ have the same number of blocks.


We show now that $R_1\setminus z=R_2\setminus z$. We have $\pi(\theta(w))\setminus z=\pi(\theta(w)\setminus z)=\pi(\theta(w\setminus a))=\evac(\pi(w\setminus a))=\evac(\Delta(\pi(w)))=\evac(\pi(w))\setminus z$, 
using Lemma \ref{removing} (iii) and (ii) for respectively the first and penultimate equality and (\ref{evac-z}) for the last one. The second equality has been proved above.

We now show that $\Delta(R_1)=\Delta(R_2)$. By Lemma \ref{removing} (ii), $\Delta(R_1)=\Delta(\pi(\theta(w)))=\pi(\theta(w)\setminus a)$. This is equal to $\pi(\theta(w\setminus z))$. By the above displayed equation, this is $\evac(\pi(w\setminus z))$. By Lemma \ref{removing} (iii), this is equal to $\evac(\pi(w)\setminus z)$ and finally, by Corollary \ref{Deltaev}, to $\Delta(\evac(\pi(w)))=\Delta(R_2)$.

2. Suppose now that $a$ does not appear in $w$. Then by induction on the cardinality of the alphabet, we have $\pi(\theta_{A\setminus a}(w))=\evac_{A\setminus a}(\pi(w))$. Thus, applying $i_a$ on both sides, using (\ref{theta-A-a}) and (\ref{evac-A-a}), and noting that $i_a$ commutes with $\pi$ (the latter is defined on each alphabet), we obtain the theorem. 

If $z$ does not appear in $w$, the argument is similar.
\end{proof} 

\section{Ordering columns}\label{order}

Following \cite{LS2}, there is a natural order on columns, as follows: $\gamma_1\leq \gamma_2$ if they are nonempty and if there is a 
tableau having the two columns $\gamma_1$ and $\gamma_2$, from left to right. For the empty column $\emptyset$, we define 
$\gamma \leq \emptyset$ for any column.
For example, looking at Figure $\ref{tab}$, and viewing columns as decreasing words, we see that $dba\leq ba \leq c$.

Equivalently also, $\gamma_1\leq \gamma_2$ if and only if there is a regressive injective mapping from $\gamma_2$ 
into $\gamma_1$ (a function $f$ is {\it regressive} if $f(x)\leq x$). Note that this order extends the order of $A$, and also
the reverse inclusion order of the subsets of $A$ \cite{LS2}. 

This order on columns is compatible with the action, as follows.
\begin{proposition}\label{increasing} (i) For each column $\gamma$ and each word $w$, one has $w\cdot 
\gamma \leq  \gamma$.

(ii) For any columns $\gamma_1,\gamma_2$, and each word $w$, $\gamma_1\leq \gamma_2$ implies 
$w\cdot\gamma_1\leq w\cdot\gamma_2$.
\end{proposition}

The next lemma is due to Bokut, Chen, Chen, Li (\cite[ Lemma 4.1]{BCCL}), in a formulation communicated to us by Darij Grinberg; moreover,  Lemma \ref{LL} is due to him, together with the proof of the second part of Proposition \ref{increasing}, which simplifies our first version.

For any column $\gamma$ and any letter $a \in A$, define $L_a(\gamma)$ to be the number of letters $\leq a$ in $\gamma$. 

\begin{lemma}\label{L} Two columns $\gamma_1$ and $\gamma_2$ over $A$ satisfy $\gamma_1 \leq \gamma_2$ if and only if each $a \in A$ satisfies $L_a(\gamma_1) \geq L_a(\gamma_2)$. 
\end{lemma}

\begin{lemma}\label{LL}
Let $A$ be an alphabet, $\gamma$ a column over $A$ and $x, a\in A$. Let $y=\max(\ell\in\gamma | \ell\leq a)$, if this set is nonempty, and otherwise, let $y=-\infty$ (smaller than any element in $A$). Then:
\begin{enumerate}
\item[i)] if $a<x$, then $L_a(x\cdot\gamma)=L_a(\gamma)$;

\item[ii)] if $y<x\leq a$, then $L_a(x\cdot\gamma)=L_a(\gamma)+1$.

\item[iii)] if $x\leq y$, then $L_a(x\cdot\gamma)=L_a(\gamma)$;
\end{enumerate}
\end{lemma}

\begin{proof}
i) In this case, $x$ does not bump any $\ell\leq a$ in $\gamma$. Therefore, the number of letters $\leq a$ remains the same.

ii) In this case, $x$ is either going to bump a letter $>a$ or will be added at the top of $\gamma$. In either cases, because $x\leq a$, $x$ is added to the count of letters $\leq a$. Therefore $L_a(x\cdot\gamma)=L_a(\gamma)+1$.

iii) In this case, $x$ will bump a letter that is $\leq a$. The number of letters $\leq a$ remains the same.
\end{proof}

\begin{proof}[Proof of Proposition \ref{increasing}] It is enough to prove both properties when $w=x\in A$. 

(i) We refer to the definition of the column insertion of $x$ into $
\gamma$ in Section \ref{insert}. In the first case, $x\cdot \gamma$ contains $\gamma$ and the result follows. In the 
second case, we have, viewing columns as decreasing words, $\gamma=uyv$ and $x\cdot\gamma=uxv$, with $y\in A$ 
and $x\leq y$; the result follows.

%

(ii) 
Let $a\in A$. Note that by Lemma \ref{LL}, $L_a(x\cdot\gamma_i)\geq L_a(\gamma_i)$. Using the fact that $\gamma_1\leq \gamma_2$ and Lemma \ref{L}, if $L_a(x\cdot\gamma_2)=L_a(\gamma_2)$ we obtain
\begin{align*}
L_a(x\cdot\gamma_2)=L_a(\gamma_2)\leq L_a(\gamma_1) \leq L_a(x\cdot\gamma_1).
\end{align*}
In the same way, if $L_a(x\cdot \gamma_1)> L_a(\gamma_1)$ we obtain
\begin{align*}
L_a(x\cdot\gamma_2)\leq L_a(\gamma_2)+1 \leq L_a(\gamma_1)+1\leq L_a(x\cdot\gamma_1).
\end{align*}
There remains only one case to verify:  $L_a(x\cdot \gamma_1)= L_a(\gamma_1)$ and $L_a(x\cdot\gamma_2)=L_a(\gamma_2)+1$. Let $y_1=\max(\ell\in\gamma_1 | \ell\leq a)$ and $y_2=\max(\ell\in\gamma_2 | \ell\leq a)$, with the same convention as for $y$ in Lemma \ref{LL}.

We know that $L_a(\gamma_1)\geq L_a(\gamma_2)$. If we have strict inequality, then $L_a(\gamma_1)\geq L_a(\gamma_2)+1$, hence $L_a(x\cdot \gamma_1)= L_a(\gamma_1)\geq L_a(\gamma_2)+1=L_a(x\cdot\gamma_2)$ and we are done.

Thus we may assume that $L_a(\gamma_1)= L_a(\gamma_2)$; then the height of $y_1$ in $\gamma_1$ is equal to the height of $y_2$ in $\gamma_2$, or they are both $-\infty$. Therefore, because $\gamma_1\leq\gamma_2$, we have $y_1\leq y_2$. Now, using Lemma \ref{LL}, $L_a(x\cdot\gamma_2)=L_a(\gamma_2)+1$ implies that $y_2 < x \leq a$. Thus $y_1<x\leq a$. We obtain by Lemma \ref{LL} (ii) that $L_a(x\cdot\gamma_1)=L_a(\gamma_1)+1$ and
\begin{align*}
L_a(x\cdot\gamma_1)=L_a(\gamma_1)+1=L_a(\gamma_2)+1=L_a(x\cdot\gamma_2).
\end{align*}
This conclude the proof.
\end{proof}

\section{$J$-relations on the stylic monoid}

\subsection{$J$-triviality}

Recall that a monoid $M$ is called {\it $J$-trivial} if for any elements $u,v\in M$ such that $MuM=MvM$, one has $u=v$. 

\begin{theorem}\label{J} $\Styl(A)$ is a $J$-trivial monoid.
\end{theorem}

\begin{proof} We mimick the proof of Proposition 4.15 in \cite{P}. Suppose that $u,v$ are words such that $M\mu(u)M=M\mu(v)M$, with $M=\Styl(A)$. Then for some words $x,y$, $v\equiv_{styl} xuy$. For any column $\gamma$, we have by 
Proposition \ref{increasing}, $\gamma\geq y\cdot\gamma$, thus $u\cdot\gamma \geq uy\cdot\gamma \geq xuy\cdot 
\gamma=v\cdot\gamma$. Symmetrically, $v\cdot\gamma \geq u\cdot\gamma $. Thus $v\cdot\gamma = u\cdot\gamma 
$. This implies that $u\equiv_{styl} v $ and $\mu(u)=\mu(v)$.
\end{proof}

In a $J$-trivial monoid, one defines the {\it $J$-order} $\leq_J$ by: $u\leq_J v $ if and only if $u\in MvM$. We study this order below.


\subsection{Left $N$-insertion}\label{left}

We describe now an algorithm which constructs, given a letter $x$ and an $N$-tableau $T$, an $N$-tableau denoted $x\to T$, and which will be shown to correspond to left multiplication by $x$ in the stylic monoid. This will serve us to prove that the $J$-order is graded (Theorem \ref{graded}). 

Let the rows of $T$ be $R_1,\ldots,R_k$ (from the lowest one to the highest), which we also view as subsets of $A$. Let 
$p_i=\min(R_i)$, the leftmost element in the row $R_i$; in particular, $p_1$ is the minimum of all elements in $T$. 
For each $i=1,\ldots,k$, let $y_i$ be the smallest element in $R_i$ which is greater than $x$, if it exists; we write $y_i=\emptyset$ if it does not exist, and $y_i\neq \emptyset$ to express that it exists. Define also $r$ to be the largest $i$ such that $x\in R_i$; if no such $i$ exists, we put $r=0$. 

\noindent Case 1: if $x<p_1$, that is, $x$ is smaller than any element in $T$, then $x\to T$ is obtained by replacing $R_1$ by $R_1\cup x$. 

\noindent Case 2: if $x$ is equal to some $p_i$, that is, if $x$ appears in the first column of $T$, then $(x\to T)=T$.

\noindent Case 3: we assume now that we are not in Case 1 nor 2. Then we have $x>p_1$. 

\noindent Subcase 3.1: if $x>p_k$, we let $t=k+1$ and $R_{k+1}$ be a new empty row.

\noindent Subcase 3.2: if $x\leq p_k$, we let $t$ be minimum with $x\leq p_t$. Then $t\leq k$, and $x<p_t$ since we are not in Case 2.

In both subcases $x\notin R_t$. Hence, 
since any element appearing in a row of an $N$-tableau also appears in lower rows, 
we must have $r<t$. Moreover, in both subcases, $p_{t-1}<x$.

In Case 3 (both subcases), $x\to T$ is obtained from $T$ by performing the two following operations (which commute): 

\noindent Step (i): add $x$ to the rows 
$R_{r+1},\ldots,R_t$ (which we call the {\it active rows}, since only these rows are modified); 

\noindent Step (ii): for $i$ satisfying $r+2\leq i\leq t$, 
remove 
$y_i$ from $R_i$ if $\emptyset\neq y_i=y_{i-1}\neq \emptyset$. 

See Figure \ref{leftInsertion} for an example: $x=d$, $r=1$, $t=4$, the 
active rows are $R_2,R_3,R_4$, $y_4=f,y_3=f,y_2=e$, hence $y_4$ disappears after left insertion of $d$, and $d$ is added in rows 2,3,4.

\begin{figure}
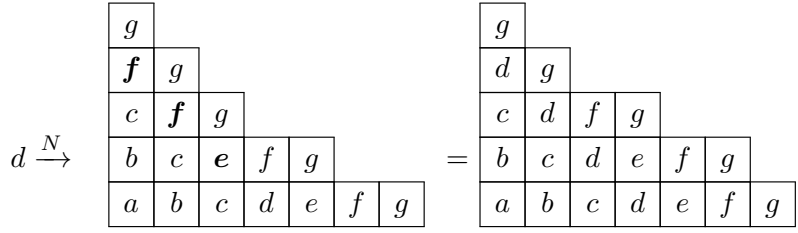

\begin{ytableau}
\none[]&\none[]&g\\
\none[]&\none[]&\boldsymbol f&g\\
\none[]&\none[]&c&\boldsymbol f&g\\
\none[d\xrightarrow{N}]&\none[]&b&c&\boldsymbol e&f&g\\
\none[]&\none[]&a&b&c&d&e&f&g
\end{ytableau}
\begin{ytableau}
\none[]&g\\
\none[]&d&g\\
\none[]&c&d&f&g\\
\none[=]&b&c&d&e&f&g\\
\none[]&a&b&c&d&e&f&g
\end{ytableau}
\caption{Left insertion}\label{leftInsertion}
\end{figure}

For later use, define $Y(x,T)=\{y \mid \exists i, r+1\leq i\leq t, y=y_i\neq\emptyset\}$.
If $Y(x,T)$ is empty, Step (ii) in Case 3 of the algorithm is empty.
If $Y(x,T)$ is nonempty, let 
$s$  be the largest $i$ such that $r+1\leq i\leq t$, and that $y_i$ exists. Then $Y(x,T)=\{y_i, i=r+1,\ldots,s\}$ and Step (ii) of Case 3 is restricted to the $i$'s satisfying $r+2\leq i \leq s$.

%
One notes also that if $t\leq k$, then $p_t>x$, $s=t$, and $y_t=p_t$.

\begin{proposition}\label{xtoTN}  $x\to T$ is an $N$-tableau. 
\end{proposition}

We begin by a simple lemma, whose proof is left to the reader. 

\begin{lemma}\label{EE'} Let $r< s$, and let $E_{r+1}\supseteq E_{r+2}\supseteq\ldots \supseteq E_s$ be a decreasing sequence of subsets of a totally ordered set, with 
minima $y_{r+1},\ldots,y_s$. Define $E'_{r+1}=E_{r+1}$, and for $i=r+2,\ldots,s$, $E'_i=E_i$ if $y_i\neq y_{i-1}$, and 
$E'_i=E_i\setminus y_i$ if $y_i=y_{i-1}$. Then $E'_{r+1}\supseteq E'_{r+2}\supseteq\ldots \supseteq E'_s$.
\end{lemma}

\begin{proof}[Proof of Proposition \ref{xtoTN}]
The only nontrivial case to consider is Case 3.
Recall that the sequence of sets $R_i$ is by definition decreasing. Denote by $R'_1,R'_2,\ldots$ the rows of $x\to T$. We verify first that this sequence of sets is decreasing. It 
is enough to show it separately for the three sequences of subsets $R'_i\cap \{c\in A\mid c<x\}$, $R'_i\cap \{x\}$ and $R'_i\cap 
\{c\in A\mid c>x\}$. For the first sequence, it follows from the equality $R'_i\cap \{c\in A\mid c<x\}=R_i\cap \{c\in A\mid c<x\}$. For the second, it is by construction the sequence of $t$ sets $\{x\}$, followed by empty sets. 

For the third sequence, suppose first that $Y(x,T)$ is empty; then Step (ii) is empty, and we have $R'_i\cap \{c\in A\mid c>x\}=R_i\cap \{c\in A\mid c>x\}$; this implies that the sequence is decreasing. Suppose now that $Y(x,T)$ is nonempty. Let $E_i=R_i\cap 
\{c\in A\mid c>x\}$ and $E'_i=R'_i\cap 
\{c\in A\mid c>x\}$; then by construction, for $r+2\leq i\leq s$, the sets $E_i,E'_i$ satisfy the hypothesis of Lemma \ref{EE'}, so 
that $E_{r+1}=E'_{r+1}\supseteq E'_{r+2}\supseteq\ldots \supseteq E'_s$. Moreover $E_i=E'_i$ for $1\leq i\leq r+1$ and 
$E'_i=E_i=\emptyset$ for $i\geq s+1$. Thus the sequence of sets $E'_i$ is decreasing.

We show now that the minima of $R'_i$ strictly increase. This follows from the fact that $p_i=\min(R_i)=\min(R'_i)$, except if $i=t$, in which case $\min(R'_t)=x$. Then the property follows from $p_1<\ldots<p_{t-1}<x<p_t<p_{t+1}<\ldots$ ($p_t,p_{t+1},\ldots$ may not exist, in which case the sequence of inequalities stops at $x$).
\end{proof}

Recall that $N$-tableaux correspond bijectively to elements in $\Styl(A)$, and that we denote by $RW(T)$ the row-word of an 
$N$-tableau $T$: one has $T=N(RW(T))$.

\begin{theorem}\label{left-insert} Let $T$ be an $N$-tableau and $x$ a letter. Then $(x\to T)=N(xRW(T))$.
\end{theorem}

In other words, left multiplication by $x$ in the stylic monoid corresponds to the left insertion into $N$-tableaux; similarly, 
we already know that right multiplication by $x$  corresponds to right $N$-insertion.

\begin{remark} It was suggested by the referee that $x\to T$ may be also computed by the following algorithm: do Schensted left insertion of $x^n$ into $T$ for $n$ large enough, obtaining a tableau $S$, and then replace each row of $S$ by its underlying subset. 
\end{remark}

We need several lemmas. 
Recall that $\D$ has been defined in Section \ref{delta}. 

\begin{lemma}\label{filtrationx} Let $R_1\supseteq\ldots\supseteq R_k$ be an $N$-filtration, and $x\in A$. Then one has the sequence of inclusions $R_1\cup 
x\supseteq 
\D_{R_2\cup x}(R_1)\supseteq \ldots\supseteq\D_{R_k\cup x}(R_{k-1})\supseteq\D_{R_{k+1}\cup x}(R_{k})$ (with $R_{k+1}=\emptyset$), and this sequence is an $N$-filtration (the last set may be empty, in which case it is removed). 
\end{lemma}

An explanation of this technical lemma may be as follows: if $R_i$ are the rows from an $N$-tableau, then one may prove that $R_i=\D_{R_i}(R_{i-1})$. Now, the deformation of the latter expression, as it appears in the lemma, gives the rows of the $N$-tableau 
corresponding to left multiplication by $x$. After this lemma, the lemma below computes these deformations.

\begin{proof} 1. Since $\D_{R_2\cup x}(R_1)$ is a subset of $R_2\cup x$, the first inclusion follows. 

2. Let $i=3,\ldots,k+1$, and $d\in \D_{R_i\cup x}(R_{i-1})$. Then there exists $c\in R_{i-1}$ such that 
$d=c^\uparrow_{R_i\cup x}\neq 1$; take $c$ maximum. Since $c\in R_{i-1}$, we have $c\in R_{i-2}$. If 
$c^\uparrow_{R_{i-1}\cup x}=d$, then $d\in \D_{R_{i-1}\cup x}(R_{i-2})$. 

Otherwise, we have either $1\neq c^\uparrow_{R_{i-1}\cup x}\neq d$, or $1=c^\uparrow_{R_{i-1}\cup x}$. In the first case, 
since $c<d$, and $d\in R_{i}\cup x$ hence $d\in R_{i-1}\cup x$, there exists $z\in R_{i-1}\cup x$ such that $c<z<d$; by maximality of $c$, we must have $z=x$, and thus 
$c<x$, hence $c^\uparrow_{R_{i}\cup x}\leq x=z<d=c^\uparrow_{R_i\cup x}$, a contradiction.

Thus we have $1=c^\uparrow_{R_{i-1}\cup x}$, which means that $c\geq \max(R_{i-1}\cup x) $; but $R_i\cup x\subseteq 
R_{i-1}\cup x$, hence $c\geq \max(R_{i}\cup x )$, hence $c^\uparrow_{R_i\cup x}=1$, a contradiction too.

From all this, the inclusion $\D_{R_i\cup x}(R_{i-1})\subseteq \D_{R_{i-1}\cup x}(R_{i-2})$ follows.

3. We show now that the sequence of minima is strictly increasing. If $x\leq \min(R_1)$, then $x=\min (R_1\cup x)$; 
therefore, for any $c\in R_1$, $c\geq x$ and $c^\uparrow_{R_2\cup x}>c\geq x$, from which follows that $
\min(\D_{R_2\cup x}(R_1))>x=\min (R_1\cup x)$. On the other hand, if $x>\min(R_1)$, then since $\D_{R_2\cup x}
(R_1)\subseteq R_2\cup x$, we have $\min(\D_{R_2\cup x}(R_1))\geq \min(R_2\cup x)>\min (R_1)$ (because $x,
\min(R_2)>\min (R_1)$)  $=\min(R_1\cup x)$.

Now, let $i=3,\ldots,k$. We show that: ($\ast$)  $\min(\D_{R_{i-1}\cup x}(R_{i-2}))< \min(\D_{R_{i}\cup x}(R_{i-1}))$. Let 
$d=\min(\D_{R_i\cup x}(R_{i-1}))$. Then there exists $c\in R_{i-1}$
such that $d= c^\uparrow_{R_i\cup x}$. It follows from the hypothesis that $R_{i-1}\subseteq R_{i-2}$ and that
$\min (R_{i-1})>\min(R_{i-2})$; thus there exists $b\in R_{i-2}$ such that $c=b^\uparrow_{R_{i-1}}$. If 
$c=b^\uparrow_{R_{i-1}\cup x}$, then $c\in \D_{R_{i-1}\cup x}(R_{i-2})$ and we deduce ($\ast$), since $c<d$. If on the 
contrary, $c\neq b^\uparrow_{R_{i-1}\cup x}$, then $b^\uparrow_{R_{i-1}\cup x}=x$ and we must have $b<x<c$ (otherwise 
by $c=b^\uparrow_{R_{i-1}}$, we have $c=b^\uparrow_{R_{i-1}\cup x}$); now $x<c<d$, hence we deduce ($\ast$) too.

If $\D_{R_{k+1}\cup x}(R_{k})$ is nonempty, then $x>\min(R_k)$, $\D_{R_{k+1}\cup x}(R_{k})=\{x\}$, and its  minimum is $x$; since $\min(R_{k-1})<\min(R_k)$, we have $\min(R_k)\in \D_{R_{k}\cup x}(R_{k-1})$, thus the minimum of this latter set is $<x$.
\end{proof}

\begin{lemma}\label{RxS} Let $\emptyset \neq R\subseteq S\subseteq A$, and $x\in A$. Let $m_R$ (resp. $m_S$) be the minimum of $R$ 
(resp. $S$) and assume that $m_S<m_R$. Define, if it exists, $y_R$ (resp. $y_S$) to be the smallest element in $R$ 
(resp. $S$) which is greater than $x$. One has:
\begin{enumerate}
\item
If $x\leq m_S$, then $\D_{R\cup x}(S)=R$. 
\end{enumerate}
Suppose now that $m_S< x$. Then one has:
\begin{enumerate}
\item[(2)] If $x\in R$ and $x\in S$, then $\D_{R\cup x}(S)=R$.
\item[(3)] If $x\notin R$ and $x\in S$, then $\D_{R\cup x}(S)=R\cup x$.
\item[(4)] If $x\notin R$, $x\notin S$, and if either 
$y_R=y_S=\emptyset$, or
$y_R=\emptyset$ and $y_S\neq \emptyset$, or
$\emptyset\neq y_R\neq y_S\neq \emptyset$, then $\D_{R\cup x}(S)=R\cup x$.
\item[(5)] If $x\notin R$, $x\notin S$, and if $\emptyset\neq y_R= y_S\neq \emptyset$, then $\D_{R\cup x}(S)=(R\cup x)\setminus y_R$.
\end{enumerate}
\end{lemma}

\begin{proof} We use several times the fact that
$R=\D_R(S)$ (which follows from $R\subseteq S$ and $m_S<m_R$).

If $x\leq m_S$, then for any $c\in S$, one has $c\geq x$, hence $c^\uparrow_{R\cup x}=c^\uparrow_R$; therefore $
\D_{R\cup x}(S)=\D_{R}(S)=R$, which proves (1). 

Assume now that $m_S< x$. We first show that in each of the cases (2) to (5), $x\in \D_{R\cup x}(S)$. Indeed, since $m_S<x$, there is some $c$ in $S$ which is $<x$, and we choose $c$ maximum; then the open interval $]c,x[$ does not intersect $S$, so does not intersect $R\cup x$ either; thus
$c^\uparrow_{R\cup x}=x$ and therefore $x\in\D_{R\cup x}(S)$.

Now let $d\in R$, with $d\neq x$ and $d\neq y_R$. We show that in each of the cases (2) to (5), $d\in \D_{R\cup x}(S)$. We have $d\in \D_R(S)$, hence there is some $c\in S$ such that $d=c^\uparrow_R$. If $x$ is not 
between $c$ and $d$, then $d=c^\uparrow_{R\cup x}$. Otherwise we have $c<x<d$, so that $y_R$ exists, and by our assumption, $d>y_R$. 
Then, since $R\subseteq S$, $y_S$ exists too, $y_S\leq y_R$, and there is some $c'$
in $S$ such that $x<c'<d$ and we choose $c'$ maximum; then $d=c'^\uparrow_{R\cup x}\in \D_{R\cup x}(S)$.

Note also that $\D_{R\cup x}(S)\subseteq R\cup x$ (because one has $\D_{B}(C)\subseteq B$), so that in 
the three cases (2) to (4), the left-hand side of the equality to be proved is contained in the right-hand side.

We now complete the proof in each case.

\begin{enumerate}
\item[(2)] We have $\D_{R\cup x}(S)=\D_{R}(S)=R$.

\item[(3)] We have $y_R=x^\uparrow_{R\cup x}\in 
\D_{R\cup x}(S)$ since $x\in S$.

\item[(4)] Note that, since $R\subseteq S$, and if $y_R$ exists, then $y_R,y_S$ both exist, and $y_S\leq y_R$. 
Thus either $y_R$ does not exist, which completes this case; or $y_R,y_S$ both exists and $y_S< y_R$ (by the assumption $y_R\neq y_S$), so that there is some $c\in S$ 
such that $x<c<y_R$, and we choose $c$ maximum; then $y_R=c^\uparrow_{R\cup x}\in 
\D_{R\cup x}(S)$.

\item[(5)]  
We show that $y_R\notin \D_{R\cup x}(S)$. Indeed, otherwise, there is some $c\in S$ such that $c<y_R$, and that $]c,y_R[$ does not intersect $R\cup x$. Then $c\neq x$ since $x\notin S$; and we cannot have $c>x$ since $y_S=y_R$. Thus we must have $c<x$, but then $]c,y_R[$ intersects $R\cup x$, a contradiction.
\end{enumerate}
\end{proof}

\begin{proof}[Proof of Theorem \ref{left-insert}] Let $T$ be an $N$-tableau with rows $R_1,\ldots,R_k$, viewed as subsets, with respective minima $p_1,\ldots,p_k$.

Suppose that we are in Case 1: $x<p_1$. Then it is apparent that in the right $N$-algorithm applied to $xRW(T)$, $x$ will appear at the first step in the first row, and the other steps will not involve $x$; hence $N(xRW(T))$ is obtained from $N(RW(T))$ by adding $x$ in the first row, and therefore $N(xRW(T))=(x\leftarrow T)$, since we are in Case 1. 

Suppose now that we are in Case 2. Then $x$ appears in the first column of $T$. Note that if $v$ is a decreasing word containing $x$, then $xv\equiv_{styl}v$, by Lemma \ref{aua}. Since $RW(T)\equiv_{styl} CW(T)$, as follows from Section \ref{plactic} and Proposition \ref{plax-styl}, we obtain that $xRW(T)\equiv_{styl} RW(T)$. Hence $N(xRW(T))=N(RW(T))=(x\to T)$.
%

We assume now that we are in Case 3. Define $u_1,\ldots,u_k$ by $\sigma(R_i)=u_i$ (the function $s$ is defined in Section \ref{delta}). Then $RW(T)=u_k\cdots u_1$. By Lemma \ref{filtrationx},
 $$S_1:=R_1\cup x\supseteq 
S_2:=\D_{R_2\cup x}(R_1)\supseteq \ldots\supseteq S_{k+1}:=\D_{R_{k+1}\cup x}(R_{k})$$
(where $R_{k+1}=\emptyset$) is an $N$-filtration, $\mathcal F$ say. It corresponds to the $N$-tableau $T'$ whose row-word is $(\prod_{i=k}^{i=1}\sigma(\D_{R_{i+1}\cup x}(R_{i})))\sigma(R_1\cup x)$. By Lemma \ref{deltaN} 
(since $\delta(x)=1$), this word is congruent modulo $\equiv_{styl}$ to $\delta(xu_k\ldots u_1)\sigma(R_1\cup x)$. This latter word is congruent to $xRW(T)$
by Lemma \ref{wdeltaSupp}.
Thus it is enough to show that for $i=1,\ldots,k+1$, $S_i=R'_i$, where the $R'_i$ 
are the rows of $x\to T$, with $R'_{k+1}$ possibly empty. We do this by following the algorithm giving $x\to T$, at the beginning of the section, and in particular using the notations there.

Since we are in Case 3, we have $x>p_1$, and either $t=k+1$, $p_k<x$ and $R_t$ empty, or $t\leq k$ and $x<p_t$. 
We have also $r< t\leq k+1$ and $p_{t-1}<x$, and if $s$ exists, then $r< s\leq t$. Next, for $i=2,\ldots,t$, we have $p_{i-1}\leq p_{t-1}<x$, so that we may apply Lemma \ref{RxS} (2), (3), (4), and (5) to $S=R_{i-1}, R=R_i$.
\begin{enumerate}
\item
If $r=0$, then $1=r+1\leq t$, hence $R'_1=R_1\cup x$; if $r\geq 1$, then $x\in R_1$, hence $R'_1=R_1=R_1\cup x$ too. Thus $S_1=R'_1$.
\item
Let $i=2,\ldots,r$; then 
$x\in R_{i-1}, x\in R_i$, so that by Lemma \ref{RxS} (2), 
$S_i=\D_{R_i\cup x}(R_{i-1})=R_i=R'_i$. 
\item
Now let $i=r+1$; then 
$x\in R_r=R_{i-1}, x\notin R_{r+1}=R_{i}$, so that by Lemma \ref{RxS} (3), $S_i=\D_{R_i\cup 
x}(R_{i-1})=R_i\cup x=R_{r+1}\cup x=R'_i$. 
\item
Suppose first that $s$ does not exist. Then for $i=r+1,\ldots,t$, $y_i$ does not exist. Let $i=r+2,\ldots,t$. Then it follows from Lemma \ref{RxS} (4) that $S_i=\D_{R_i\cup x}(R_{i-1})=R_i\cup x =R'_i$.

Suppose now that $s$ exists.
Let $i=r+2,\ldots s$; then $y_i,y_{i-1}$ exist, 
$x\notin R_{i-1}, x\notin R_i$, so that 
$S_i=\D_{R_i\cup x}(R_{i-1})=R_i\cup x $ or $(R_i\cup x)\setminus y_i$, depending on $y_i\neq y_{i-1}$ or $y_i =
y_{i-1}$ (
by Lemma \ref{RxS} (4) and (5)) $=R'_i$. 

If $i=s+1\leq t$, then $y_i$ does not exist, $y_{i-1}$ exists, $x\notin R_i$, $x\notin R_{i-1}$, so that $S_i=\D_{R_i\cup x}(R_{i-1})=R_i\cup x=R'_i$, by Lemma \ref{RxS} (4)).

Now let $i=s+2,\ldots,t$. Then 
$x\notin R_{i-1},x\notin R_i$, $y_{i-1}, y_i$ do not exist, so that $S_i=\D_{R_i\cup x}(R_{i-1})=R_i\cup x$ (by Lemma \ref{RxS} (4)); hence $S_i=R'_i$.

\item Finally, suppose that either $t=k+1, R_t$ empty and $i=k+1$, or $t+1\leq i\leq k+1$. In the first case,
$S_{k+1}=\D_{x}(R_k)=x$ (since $p_k<x$) $=R'_{k+1}$. In the second case, 
$x< p_t\leq p_{i-1}$, so that by Lemma \ref{RxS} (1), we have $S_i=\D_{R_i\cup x}(R_{i-1})=R_i=R'_i$.
\end{enumerate}
\end{proof}

\subsection{Grading of the $J$-order}

\begin{figure}
\centering
\includegraphics[scale=0.7]{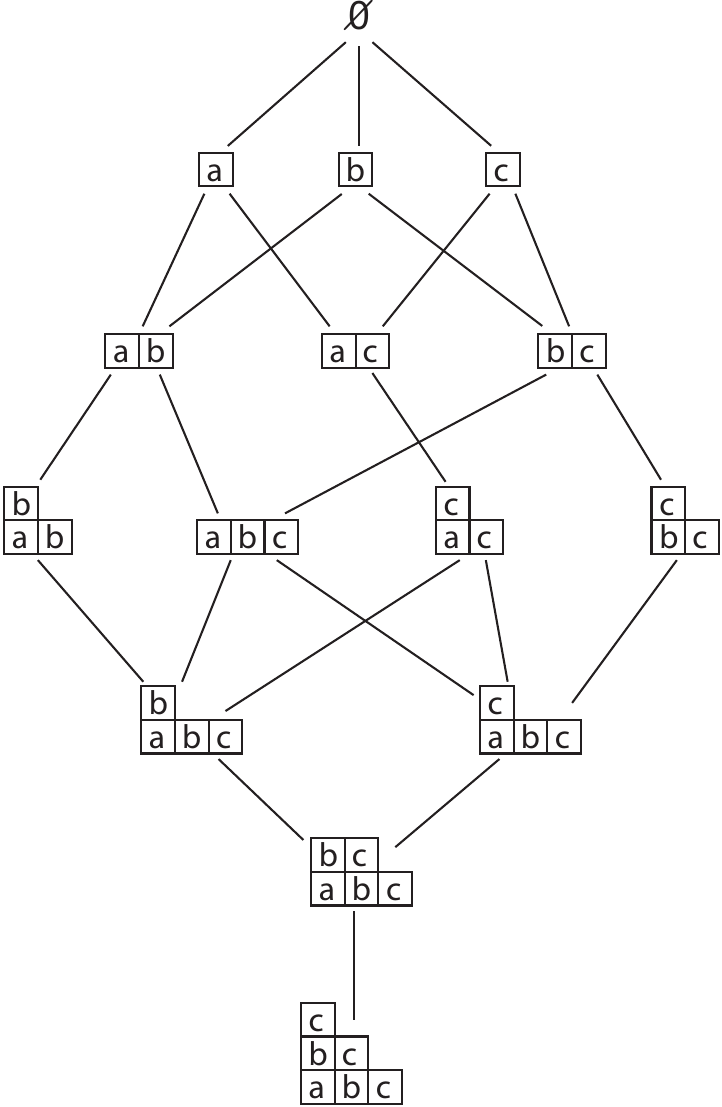} 
\caption{$J$-order for $n=3$}
\end{figure}

A finite poset $P$ is {\it graded} if there is a function $\rk:P\to \N$ such that: if $x<y$ in $P$, then $\rk(x)<\rk(y)$, and if moreover $y$ covers $x$, then $\rk(y)=\rk(x)+1$. The function $\rk$ is called the {\it rank function}. If $P$ has a minimum $\hat 0$ and a maximum $\hat 1$, we may assume that $\rk(\hat 0)=0$; let $N=\rk(\hat 1)$. We then call the function $P\to \N, x\mapsto N-\rk(x)$ the {\it co-rank function}.

\begin{theorem}\label{graded} The $J$-order in $\Styl(A)$ is graded. The co-rank of an element is given by the number of boxes in its 
$N$-tableau. 
\end{theorem}

Although the co-rank function is easy to describe, we do not know a direct criterion that characterizes the $J$-order directly on the $N$-tableaux.

We need some preparation. For the first result, recall that the {\it shape} of a {\it semi-standard tableau} $T$, denoted 
$\lambda(T)$, is the integer partition whose parts are the row lengths of $T$. Each integer partition is classically denoted by the decreasing sequence of its parts. The {\it Young order} on integer partitions is obtained by the rule $(a_1,a_2,\ldots)\leq_{Young} (b_1,b_2,\ldots) \Leftrightarrow \forall i, a_i\leq b_i$ (where $a_1\geq a_2\geq \ldots, b_1\geq b_2\geq \ldots$, and where sufficiently many 0's are added to the sequences).

\begin{proposition}\label{Young} Let $T$ be an $N$-tableau, $x$ a letter, and let $S=(T\leftarrow x)$ (resp. $S=(x\to T)$).  If $S\neq T$, then $\lambda(T)<_{Young}\lambda(S)$.
\end{proposition}

It follows that the function $S\mapsto \lambda (S)$ is strictly increasing, from the set of $N$-tableaux with the $J$-order into the set of integer partitions with the Young order.

\begin{proof} Let $S=T\leftarrow x$ and suppose that $S \neq T$. It follows directly from the algorithm of right $N$-insertion that, since $S\neq T$, several rows (and at least one row) of $T$ get a new element, and the other rows remain unchanged, producing $S$. Thus $\lambda(T)<_{Young}\lambda(S)$.

Recall that $\Styl(A)$ is in bijection with partitions, and also with $N$-tableaux. It follows directly from the definition of evacuation, and of Theorem \ref{evac}, that the anti-automorphism $\theta$  of $\Styl(A)$ preserves the shape of a partition, hence of the corresponding $N$-tableau. Let $y=\theta(x)$. Then the image under $\theta$ of $x\to T$ is $T\leftarrow y$, by Corollary \ref{right-insert} and Theorem \ref{left-insert}. Thus the second assertion follows from the first (it may also be seen directly on the left insertion).
\end{proof}

\begin{lemma}\label{Tc} Let $T$ be an $N$-tableau, with rows $R_1,\ldots,R_k$, and let $R_{k+1}$ denote the empty row.

(i) Let $c$ be a letter in $R_i$, with $c\neq \min(R_i)$. There exists a letter $x\in \Supp(T)$ such that the letter $c$ is bumped from $R_i$ during the right $N$-insertion $T\leftarrow x$, and no letter is added in rows $R_1,\ldots,R_i$. 

(ii) Let $c\in R_i$, such that either $c>\max(R_{i+1})$, or $i=k$ and $c\neq \min(R_k)$. Then there exist letters $x_1,\ldots x_{k+1-i}\in \Supp(T)$ such that the $k+1-i$ successive right $N$-insertions $(\ldots(T\leftarrow x_1)\cdots )\leftarrow x_{k+1-i}$ produce an $N$-tableau which is obtained from $T$ by adding $c$ in rows $R_{i+1},\ldots,R_{k+1}$, with one box labelled $c$ added at each insertion.
\end{lemma}

\begin{proof} (i) (induction on $i$)  If $i=1$, then $c\neq \min(R_1)$ implies that some letter $x<c$ is in $R_1$; we choose $x$ maximum and then in the right $N$-insertion $T\leftarrow x$, $c$ is bumped from $R_1$; the second condition holds since $x\in R_1$.

Suppose now that $i\geq 2$; since $\min(R_i)<c$, there is some letter $b<c$ in $R_i$ and we choose $b$ maximum; then the right $N$-insertion of $b$ into the row $R_{i}$ bumps $c$ from $R_i$.
Since $T$ is an $N$-tableau, $\min(R_{i-1})<\min(R_i)\leq b$ and $b\in R_{i-1}$. Thus, by induction on $i$, there exists a letter $x\in \Supp(T)$ such that during the right $N$-insertion $T\leftarrow x$, $b$ is bumped from $R_{i-1}$, and no letter is added in rows $R_1,\ldots,R_{i-1}$. It follows that during this insertion, $c$ is bumped from $R_i$. Moreover, no letter is added in $R_{i}$, since $b\in R_{i}$.

(ii) (induction on $k+1-i$) If $i=k$, then by (i) there exists $x_1\in\Supp(T)$ such that the right $N$-insertion $T\leftarrow x_1$ bumps $c$ from $R_k$, producing a new row $\{c\}$, and no letter is added in rows $R_1,\ldots,R_{k}$.

Suppose now that $i<k$. Then ($\ast$)  $c>\max(R_{i+1})\geq \min(R_{i+1})>\min(R_i)$, so that by (i) there exists $x_1\in \Supp(T)$ 
such that $T\leftarrow x_1$ adds $c$ into row $R_{i+1}$, and since $c>\max(R_{i+1})$, the algorithm stops; denote by $T'$ 
the resulting tableau, with rows $R'_j$. We have $c\in R'_{i+1}$ and either: $i+1<k$, $c>\max(R_{i+1})\geq \max(R_{i+2})
$ (since $R_{i+1}$ contains $R_{i+2}$) $=\max(R'_{i+2})$ (since $R_{i+2}=R'_{i+2}$); or: $i+1=k$ and $c\neq \min(R_k)$
as follows from ($\ast$).

It follows by induction that there exist $x_{2},\ldots,x_{k+1-i}\in\Supp(T')=\Supp(T)$ such that $(\ldots(T'\leftarrow x_2)\cdots)\leftarrow x_{k+1-i}$ adds $c$ in rows $R'_{i+2},\ldots,R'_{k+1}$, and such that each insertion adds exactly one box. This ends the proof.
%
\end{proof}

\begin{proof}[Proof of Theorem \ref{graded}] For $w\in A^*$, let $\nu(w)$ denote the number of boxes in the $N$-tableau $N(w)$; this function is compatible with the stylic congruence, hence induces a function $\nu$ on $\Styl(A)$; we show that it is the co-rank function on the poset $\Styl(A)$ with the $J$-order.

Since $\Styl(A)$ is generated by the letters, the $J$-order is defined by the following rule: for $u,v\in \Styl(A)$, $u\leq_J v$ if and only if there exist elements $u_0,\ldots,u_n$ in $\Styl(A)$ such that $u_0=u$, $u_n=v$ and for each $i=0,\ldots,n-1$, there exists a letter $x$ such that $u_i=xu_{i-1}$ or $u_i=u_{i-1}x$. Switching to $N$-tableaux, identified with elements of $\Styl(A)$, this translates into: for any $N$-tableaux $S,T$, one has $T\leq_J S$ if and only if there exist $N$-tableaux $T_0,\ldots,T_n$ such that $T_0=T$, $T_n=S$ and for each $i=0,\ldots,n-1$, there exists a letter $x$ such that $T_i=x\to T_{i-1}$ or $T_i=T_{i-1}\leftarrow x$.

We therefore deduce from Proposition \ref{Young} that for $u,v\in \Styl(A)$ such that $u<_J v$, one has $\nu(u)<\nu(v)$.

It remains to show the following result: if for two $N$-tableaux $S,T$, $T\leq_J S$, then there exists a sequence of $N$-tableaux $T_0,\ldots,T_n$ such that $T_0=T$, $T_n=S$ and for each $i=0,\ldots,n-1$, $T_{i-1}<_J T_i$ and $\nu(T_i)=\nu(T_{i-1})+1$. It is enough to prove this when $S$ is obtained from $T$ by a left or a right insertion of a letter, and even when it is a left insertion (since the anti-automorphism $\theta$ exchanges left and right insertions, and preserves the shape, hence also preserves $\nu$). We show that this left insertion $x\to T$ is equivalent to a sequence of left or right insertions, each adding one box to the shape. We argue by induction on $\nu(S)-\nu(T)$, noting that if this quantity is 0, then $S=T$.

So let $S=(x\to T)$ for some letter $x$. Referring to the definition of left insertion at the beginning of Section \ref{left}, we see that Cases 1 and 2 give immediately the result. So we may assume that we are in Case 3. There are two cases to consider: $t=k+1$ (subcase 3.1), and $t\leq k$ (subcase 3.2).

1. Suppose that $t=k+1$; then $R_{k+1}=\emptyset$ and $\min(R_k)=p_k<x$. 

a) Suppose that $Y(x,T)$ is empty. 
Then $S$ is obtained from $T$ by adding $x$ in each row $R_{r+1},\ldots, R_{k+1}$. If $r\geq 1$, we use Lemma \ref{Tc} (ii), 
with $c=x, i=r$: the left insertion $x\to T$ may be simulated by $k+1-r$ right insertions of $x$, each one increasing the 
number of boxes by 1. If $r=0$, then the hypothesis $Y(x,T)=\emptyset$ implies that $x$ is greater than each element in $T$; then $T\to x$ 
adds $x$ in the first row, and nothing else, and we are reduced to $r\geq 1$.

b) Suppose that $Y(x,T)$ is nonempty. Then $y_s\in R_s$ and: either $s<k$, and then $y_s>x\geq \max(R_{s+1})$ (the latter 
inequality since $y_{s+1}$ does not exist), and therefore $y_s>\max(R_{s+1})$; 
or $s=k$ and 
$y_s\neq \min(R_k)$ (since $y_s>x>p_{k}$). Then by Lemma \ref{Tc} (ii) (applied to $c=y_s$, $i=s$), we find a sequence of right insertions, each one of which 
adds a single box, and whose result is the $N$-tableau $T'$ obtained from $T$ by adding $y_s$ in rows $R_{s+1},
\ldots,R_{k+1}$ (and in particular the $(k+1)$-th row of $T'$ is $\{y_s\}$). Then $\nu(T')>\nu(T)$, since $s+1\leq k+1$. Now $x\to T'$ adds $x$ in rows $R_{r+1},
\ldots,R_t$ and removes the $y_s$ that were just added, together with each $y_i$ in rows $R_{r+2},\ldots,R_s$ if $y_i=y_{i-1}$; 
thus $(x\to T')=S$. We conclude by induction, since $\nu(S)-\nu(T')$ is smaller than $\nu(S)-\nu(T)$.

2. Suppose that $t\leq k$. Then $s=t$, $y_{r+1},\ldots, y_t$ exist and $y_t=p_t$; moreover  $Y(x,T)=\{y_{r+1},\ldots, 
y_t\}$. 

a) Suppose that the set $Y$ has only one element, which is $y_t$. Then (i) adds $x$ in rows $R_{r+1},\ldots, R_t$ and (ii) removes $y_t$ from the rows $R_{r+2},\ldots, R_t$; hence $\nu(S)=\nu(T)+1$ and we are done.

b) Suppose that the set $Y$ has at least two elements, and let $y_u=\max(Y\setminus y_t)$, with $u$ chosen maximum. Let $T'=(y_u\to T)$. We have $p_{t-1}<x<y_u<y_t=p_t$. Hence the left insertion $y_u\to 
T$ adds $y_u$ in rows $R_{u+1},\ldots, R_t$ and removes $y_t$ from the rows $R_{u+2},\ldots R_t$, and in particular $
\nu(T')=\nu(T)+1$. Now, for the left insertion $x\to T'$, we have $Y(x,T')=Y(x,T)\setminus y_t$. Moreover, the left insertion 
$x\to T'$: (i) adds $x$ in rows $R_{r+1},\ldots,R_{t}$; and (ii) removes $y_u$ from rows $R_{u+1},\ldots,R_t$, and from the 
rows $R_i$, $i=r+2,\ldots,u$, it removes $y_i$ if $y_i=y_{i-1}$. Thus $(x\to T')=S$, which 
settles this case by induction, because $\nu(T')=\nu(T)+1$.

\end{proof}

%
%
%

\section{Fixpoints and idempotents}

Recall that we may view columns as subsets of $A$. As such, they are ordered by inclusion.

\begin{theorem}\label{fixed} (i) Let $w\in A^*$. A column $\gamma$ is fixed by $w$ if and only if $\Supp(w)\subseteq\gamma$.

(ii)  The support of a word is the smallest fixpoint, in the inclusion order, of its action on the columns. 

(iii) The idempotents in $\Styl(A)$ are the images under $\mu$ of the strictly decreasing words. 

(iv) There are $2^{|A|}$ idempotents in $\Styl(A)$.
\end{theorem}

\begin{lemma}\label{column.column} Let $w$ be a strictly decreasing word and $\gamma $ a column. Then $\Supp(w)\subseteq \Supp(w\cdot \gamma)$.
\end{lemma}

\begin{proof} We show this by induction on the length of $w$. 
It is clear if $w$ is empty. Otherwise $w=av$, with $v$ strictly decreasing. By induction $\Supp(v)\subseteq v\cdot \gamma$. 
Since $a$ is greater than any letter in $v$, we also have $\Supp(v)\subseteq (v\cdot \gamma)_a$.
We have $(a\cdot(v\cdot\gamma))_a=(v\cdot\gamma)_a$ by Lemma \ref{idemp} (ii). Hence $\Supp(v)\subseteq (a\cdot(v\cdot\gamma))_a
\subseteq (av)\cdot \gamma=w\cdot\gamma$, and since $a\in a\cdot(v\cdot \gamma)=w\cdot \gamma$, we deduce $\Supp(w)\subseteq w\cdot \gamma$, as was 
to be shown.
\end{proof}

\begin{proof}[Proof of Theorem \ref{fixed}] (i) If $w$ is a word such that $\Supp(w)\subseteq \gamma$, $w$ fixes $\gamma$ by Lemma \ref{idemp} (i) applied iteratively. 

Conversely, let $\gamma$ be a column fixed by $w$. 
If we had $w=uav$ with $a \notin \gamma$, choose $v$ shortest possible; then $\Supp(v)\subseteq \gamma$, thus by Lemma \ref{idemp} (i), $v\cdot \gamma=\gamma$. We have 
$a\cdot \gamma \neq \gamma$, and by Proposition \ref{increasing} (i), $a\cdot \gamma < \gamma$. Therefore $(av)
\cdot\gamma<\gamma$ and finally $w\cdot \gamma =u\cdot ((av)\cdot\gamma)\leq (av)\cdot\gamma$ (by the same proposition)
$< \gamma$, and we cannot have $w\cdot \gamma=\gamma$, a contradiction.

(ii) Clear by (i).

(iii) Let $w$ be a strictly decreasing word. Then we already know that the fixpoints of $w$ are the columns containing $\Supp(w)$. 

Let $\gamma$ be any column. 
It follows from Lemma \ref{column.column} and (i), that $w\cdot\gamma$ is a fixpoint of $w$. Hence $w\cdot(w\cdot \gamma)=w\cdot \gamma$, and $w$ acts as idempotent on the columns.

It remains to prove the converse: each idempotent $e$ in $\Styl(A)$ is equivalent modulo $\equiv_{styl}$ to a strictly 
decreasing word. For this, let $w$ the strictly decreasing word whose letters are the elements in $\Supp(e)$. Then by (i) 
$e$ and $f=\mu(w)$ have the same set of fixpoints; moreover, $e,f$ are idempotent, hence their images are
contained in this set. It follows by monoid theoretical arguments that $\mu(w)=e$: indeed, for any $\gamma$, $f\cdot 
\gamma$ is in the image of $f$, hence is a fixpoint of $e$; hence $ef\cdot \gamma=f\cdot \gamma$; thus $ef=f$; similarly, 
$fe=e$; hence $e,f$ are $J$-equivalent, hence equal since $\Styl(A)$ is $J$-trivial (Theorem \ref{J}).

(iv) is clear, since the idempotents are in bijection with subsets of $A$, because two different subsets, viewed as strictly decreasing words, act differently on the empty column. 
\end{proof}

\subsection{Applications to the plactic monoid: a confluent rewriting system}


In the next result, columns are also viewed as decreasing words, and as subsets of $A$.

\begin{proposition}\label{gammadelta} Let $\gamma,\delta$ be columns. Let $\gamma'=\gamma\cdot \delta$ and 
$\delta'=(\gamma\cup\delta)\setminus (\gamma\cdot\delta)$, where this boolean operation is taken as multisets. Then

(i) $\gamma\subseteq \gamma'$;

(ii) $\gamma'\leq \gamma$;

(iii) $\gamma'=\gamma$ if and only $\delta=\delta'$ if and only if $\gamma\leq \delta$;

(iv) $\gamma'\leq \delta'$;

(v) $\gamma \delta\equiv_{plax}\gamma'\delta'$.
\end{proposition}

Note that $\delta'$ may be the empty column, in which case it is the empty word, according to our conventions.

\begin{proof} (o) Consider the tableau $T$ obtained by column insertion of the word $\gamma$ into the column $\delta$: its first column is $\gamma'=\gamma\cdot \delta$, by definition of the action on columns, and $T$ has either only one column, or two columns, and by counting letters, the second one must be $\delta'$; in particular, $\gamma'\leq \delta'$ by definition of order. Moreover $T=P(\gamma\delta)$. In particular, if $\gamma\leq \delta$, $T$ is the tableau with first column $\gamma=\gamma'$ and second column $\delta$.

(i) This is Lemma \ref{column.column}.

(ii) This follows from (i) and an observation in Section \ref{order}, relating inclusion of columns, and their order.

(iii) The first equivalence follows from the multiset union $\gamma\cup \delta=\gamma'\cup\delta'$. If $\gamma\leq \delta$, 
then $\gamma'=\gamma$ by (o). Conversely, if $\gamma'=\gamma$, then $\delta'=\delta$, and we obtain by (o)  that $
\gamma'\leq \delta'$, hence $\gamma\leq \delta$.

(iv) and (v) follow from (o) and Section \ref{insert}.
\end{proof}

\begin{theorem}\label{quadr-pres} (\cite[Theorem 4.5]{BCCL}, \cite[ Theorem 3.4]{CGM}) 

(i) The plactic monoid has the following presentation: 
it is generated by the columns, subject to the relations $\gamma\delta
=\gamma'\delta'$, for all columns $\gamma,\delta$, where $\gamma',\delta'$ are defined in Proposition \ref{gammadelta}, 
with $\delta'$ omitted if it is the empty column. 

(ii) The rewriting system on the free monoid $\C(A)^*$ given by the rules $\gamma\delta\to\gamma'\delta'$, with the same 
notations, and where one omits the rules with $\gamma\leq\delta$, is confluent. 
\end{theorem}

Recall that a {\it rewriting system} on a free monoid $C^*$, generated by rules $u\to v$, is the least reflexive and transitive 
binary relation on $C^*$ which is compatible with left and right multiplication. It is {\it confluent} if the set of 
words $w$ which
may not be rewritten (that is, do not have as factor any word $u$ which is the left part of a rule) is a set of unique
representatives of the congruence generated by this binary relation.

\begin{proof} Consider the order $\leq $ on the set $\C(A)$ of columns, and then order lexicographically the words of equal 
length in the free monoid $C(A)^*$, then order this whole free monoid first by length, then lexicographically. We obtain an 
order on $\C(A)^*$, which is not total, but suffices for our purpose. If we use a rule 
$\gamma\delta\to\gamma'\delta'$ in a word $w$, obtaining $w'$, then either $w'$ is shorter than $w$ (in case $\delta'$ is 
the empty column); or $w'$ and $w$ have the same length, and $w$ is smaller for the previous lexicographic order, since $
\gamma'<\gamma$, by Proposition \ref{gammadelta} (ii) and (iii) (because we do not have $\gamma\leq \delta$). Hence 
$w'<w$. It follows that there is no infinite chain in the rewriting rule, since each such chain decreases for the order, and remains 
in the finite set of words of bounded length.

As a consequence, each word may be rewritten into a word $\gamma_1\cdots\gamma_n$ with $\gamma_1\leq \ldots\leq 
\gamma_n$. Since tableaux form a set of representatives of the plactic monoid and by Proposition \ref{gammadelta} (v), we 
obtain the theorem.
\end{proof}

For the interested reader, note that Bokut et al. give a formula in order to compute $\gamma'$ and $\delta'$, with the 
notations of Proposition \ref{gammadelta}: see \cite[Definition 4.6 and Lemma 4.7]{BCCL}. Note also that when we write $
\gamma_1 \leq \gamma_2$, they write $\gamma_ 1 \vartriangleright\gamma_2$ (and $\gamma_1\succeq\gamma_2$ in 
\cite{CGM}).

\section{Syntacticity}\label{synt}

The syntactic monoid and congruence of a language (= subset of a free monoid) are well-known notions (see for example 
\cite{P}). As is also well-known, they immediately extend to functions from a free monoid into any set, as follows. 

Let $f:A^*\to E$, where $E$ is any set. The {\it syntactic congruence of $f$}, denoted $\equiv_f$, is defined by
$$
u\equiv_f v \Leftrightarrow (\forall x,y\in A^*, f(xuy)=f(xvy)).
$$ 
It is a (two-sided) congruence of $A^*$, that is, an equivalence relation which is compatible with the product in $A^*$. It is 
the coarsest congruence $\equiv$ of $A^*$ which is compatible with $f$, that is, satisfying $u\equiv v\Rightarrow f(u)=f(v)$.  
The {\it syntactic monoid of $f$} is the quotient monoid $M_f=A^*/\equiv_f$. One has clearly $u\equiv_fv\Rightarrow 
f(u)=f(v)$, so that $f$ induces a function $g_f:M_f\to E$ such that $f=g_f\circ \mu$, with $\mu$ the canonical monoid 
homomorphism $A^*\to M_f$.

Similarly, the {\it left synctactic congruence of $f$}, denoted by $\equiv_f^l$, and defined by
$$
u\equiv_f^l v \Leftrightarrow (\forall x\in A^*, f(xu)=f(xv)).
$$
It is a left congruence of $A^*$, that is, compatible with multiplication at the left, and one therefore obtains a left action 
of $A^*$ onto the set $A^*/\equiv_f^l$. The syntactic left congruence of $f$ is the coarsest left congruence of $A^*$ which 
is compatible with $f$.

Both quotients have a universal property with respect to $f$, which we describe only for the syntactic monoid. Consider the 
category whose objects are the triples $M,\mu,g$, where $M$ is a monoid, $\mu$ a surjective monoid homomorphism $A^*\to 
M$, and $g:M\to E$ a function, such that $f=g\circ \mu$; in this case, we say that $M,\mu,g$ (or simply $M$) {\it 
recognizes} $f$. Morphisms of the category are defined as monoid 
homomorphisms $\nu:M\to M'$ such that $\mu'=\nu\circ\mu$ and $g'\circ\nu=g$.
The triple $M,\mu,g_f$ is an object in the category, and it is a final object in the category.
In that sense, we may say that ``$M_f$ is the smallest monoid recognizing $f$". 

\begin{theorem}\label{syntac} Consider the function $f$ which associates to $w\in A^*$ the maximum length of a strictly 
decreasing subsequence of $w$; equivalently (by Schensted's theorem) the length of the first column of $P(w)$.

(i) The syntactic left congruence of $f$ is determined by: $u\equiv_f^l v$ if and only if $u\cdot \emptyset=v\cdot \emptyset$ (where $\emptyset$ is 
the empty column).

(ii) The syntactic monoid of $f$ is $\Styl(A)$, and its syntactic congruence $\equiv_f$ coincides with $\equiv_{styl}$.
\end{theorem}

\begin{lemma}\label{action1}
Let $n\geq 2$ and letters $a_n>\ldots>a_2>a_1$ and $u$ a word such that $a_n\cdots a_3a_1u$ is a strictly decreasing word. 
Viewing columns as strictly decreasing words, let $\gamma=a_n\cdots a_3a_1u, \gamma'=a_{n-1}\cdots a_1u$; then
$$a_{n-1}\cdots a_2\cdot \gamma=\gamma'.$$
\end{lemma}

\begin{proof} For $n=2$, the equality is $1\cdot a_1u=a_1u$, which is true. Suppose that $n\geq 3$. Let $v=a_1u$. By induction, 
$a_{n-1}\cdots a_3\cdot a_n\ldots a_4 a_2v=a_{n-1}\cdots a_2v$. We have $a_2\cdot a_n\cdots a_3a_1u=a_n\cdots 
a_4a_2a_1u$. By the previous equality, we obtain therefore
$$
a_{n-1}\cdots a_2\cdot \gamma=a_{n-1}\cdots a_2\cdot a_n\cdots a_3a_1u=(a_{n-1}\cdots a_3)\cdot(a_2\cdot a_n\cdots a_3a_1u)
$$
$$
=(a_{n-1}\ldots a_3)\cdot a_n\ldots a_4a_2a_1u=a_{n-1}\cdots a_2a_1u,
$$
which was to be shown.
\end{proof}

\begin{proof}[Proof of Theorem \ref{syntac}] (i) If $u\cdot\emptyset= v\cdot \emptyset$, then for any word $x$, $xu\cdot\emptyset=xv\cdot \emptyset$ and 
therefore $f(xu)=f(xv)$, since $w\cdot \emptyset$ is the first column of $P(w)$ by Proposition \ref{first}, and $f(w)$ is its length.
Thus $u\equiv_f^l v$.

Conversely, suppose that $\gamma_1=u\cdot \emptyset\neq v\cdot \emptyset=\gamma_2$. In order to show that $u,v$ are not 
equivalent modulo $\equiv_f^l$, it is 
enough to show the existence of a word $x$ such that $f(xu)\neq f(xv)$, that is: the first columns of $P(xu)$ and $P(xv)
$ have different lengths. We 
know by Proposition \ref{first} that these columns are $xu\cdot \emptyset$ and $xv\cdot \emptyset$, equivalently $x\cdot\gamma_1$ and 
$x\cdot \gamma_2$.

If the two columns $\gamma_1,\gamma_2$ have different length, we take $x=1$.
Suppose now that they have the same length. If their largest letter are distinct, we may assume that it is $a$ for $
\gamma_1$ and $b$ for $\gamma_2$ 
and $a<b$; then $b\cdot \gamma_1=\gamma_1\cup b$ and $b\cdot\gamma_2=\gamma_2$ (since $b$ appears in $
\gamma_2$) and these columns have 
different lengths: we then take $x=b$. If their largest letters are equal, we may write (for example)
$\gamma_1=a_n\cdots a_3a_1s$, $\gamma_2=a_n\cdots a_3a_2t$, with $n\geq 2$, $a_n>\cdots >a_3>a_2>a_1$, and $s,t$ of the same length; let 
$w=a_{n-1}\cdots a_2$; then by Lemma \ref{action1}, $w\cdot\gamma_1=a_{n-1}\cdots a_1s$ and $w\cdot \gamma_2=\gamma_2=a_n\cdots 
a_3a_2t$ (by Lemma \ref{idemp} (i), since $\Supp(w)\subseteq \Supp(\gamma_2)$); these two columns have distinct largest letters, and we are 
reduced to the previous case.

(ii) The argument we use now is standard in algebraic automata theory. We have $u\equiv_f v\Leftrightarrow 
(\forall x,y\in A^*, f(xuy)=f(xvy))
\Leftrightarrow (\forall y\in A^*, uy\equiv_f^l vy) 
\Leftrightarrow (\forall y\in A^*, uy\cdot \emptyset=uy\cdot \emptyset)$ (by (i))
$\Leftrightarrow (\forall \gamma \in\C(A), 
u\cdot\gamma=v\cdot\gamma) $ (since $(wy)\cdot\emptyset=w\cdot(y\cdot \emptyset)$ and since each column is of the form $y\cdot \emptyset$)
$\Leftrightarrow u\equiv_{styl}v$.
\end{proof}






\section{Appendix: a theorem of Lascoux and Sch\"utzenberger}

In \cite[Théorème 2.15 p.136]{LS1}, Lascoux and Schützenberger state that the plactic congruence is the syntactic monoid of the function $\lambda$ which associates with each word $w$ the shape (``forme immanente" in their article) of the tableau $P(w)$. Equivalently: $u\equiv_{plax} v\Leftrightarrow (\forall x,y\in A^*, \lambda(xuy)=\lambda(xvy))$ (see Section \ref{synt} for the definitions about syntacticity).
Their theorem is given without proof, and we provide a proof below, and a generalization. 

Denote by $\Lambda$ the set of all integer partitions. The {\it shape} of a tableau is the partition whose parts are the lengths of its rows.

\begin{theorem}\label{plactic-synt} The left syntactic congruence $\equiv_\lambda^l$ of the function $\lambda:A^*\to\Lambda$, which associates with each word the shape $\lambda(w)$ of the tableau $P(w)$, is the plactic congruence.
\end{theorem}

The theorem hold also for the right syntactic congruence, as follows from the application of the anti-automorphism $\theta$.

\begin{corollary}(Lascoux and Schützenberger) The plactic congruence is the syntactic congruence of the function $\lambda$.
\end{corollary}

\begin{proof} If $u\equiv_{plax}v$, then $xuy\equiv_{plax}xvy$, hence $P(xuy)=P(xvy)$, for any words $x,y$. Thus $
\lambda(xuy)=\lambda(xvy)$. Conversely, if $\forall x,y\in A^*, \lambda(xuy)=\lambda(xvy)$, then in particular $
\forall x\in A^*, \lambda(xu)=\lambda(xv)$; hence $u\equiv_\lambda^l v$ and therefore $u\equiv_{plax}v$ by 
Theorem \ref{plactic-synt}.
\end{proof}

Recall that each plactic equivalence class contains a unique representative which is a product of columns $$w=\gamma_1\cdots \gamma_N, \gamma_1\leq \cdots\leq \gamma_N,$$ where $\leq$ is the order on columns of Section \ref{order}.
This follows from Section \ref{plactic}, by considering the column-word of a tableau. We call the {\it column representative} this representative of the plactic 
class.

\begin{lemma}  Let $w$ be as above. 
Let $b\in A$ and $y$ the strictly decreasing word involving all letters $\geq b$ in $A$. Let $n\leq N$.
Write $\gamma_i=u_iv_i$, where $u_i$ involves only letters $\geq b$, and $v_i$ only letters $<b$.
Then the column representative of the plactic class of $y^nw$ is
$$
m=(yv_1)\cdots(yv_n) v'_{n+1}v'_{n+2}\cdots.
$$
\end{lemma}

\begin{proof} Note the identity $y^n\equiv_{plax}\prod_{t\in A,t\geq b}t^n$, where the product is strictly decreasing from the largest letter in $A$ until $b$: this identity is true because the two sides are the column- and row-words of the rectangular tableau with $n$ columns, all equal to $y$. Note that the left product by $y^n$ (equivalent to Schensted  left insertion of $y^n$) does not change the letters $\leq b$ in the columns; moreover the product by $b^n$ introduces a $b$ in the $n$ first columns, the product by $c^n$ (with $c$ the next letter in $A$) introduces a $c$ in them, and so on. Finally, these columns contain all the letters $\geq b$, which proves the lemma.
\end{proof}

\begin{proof}[Proof of Theorem \ref{plactic-synt}] If $w\equiv_{plax}w'$, we have for any word $x$, $xw\equiv_{plax}xw'$, and therefore $P(xw)=P(xw')$ and $\lambda(xw)=\lambda(xw')$.

Conversely, suppose that $w,w'$ are not equivalent modulo the plactic relation. Then for some $n\geq 1$, for $i=1,\ldots, 
n-1$, the $i$-th columns of $P(w)$ and $P(w')$ are equal, and their $n$-th columns differ. If their $n$-th columns have 
different heights, then $\lambda(w)\neq \lambda(w')$ and we choose $x=1$. If their heights are equal, let $a<b$ be the 
first letters distinguishing these columns, from left to right (columns being viewed as strictly decreasing words): $a$ appears in the $n$-th 
column of $w$, and $b$ in the $n$-th column of $w'$, and the letters at the left of $a$ (in $w$) and $b$ (in $w'$) in the two 
$n$-th columns are equal.

Then the plactic classes of $w$ and $w'$ have respectively columns representations of the form given in the displayed 
equation before the lemma (with primes for $w'$), and $\gamma_i=\gamma_i'$ for $i=1,\ldots,n-1$. We  may write $\gamma_i=u_iv_i=\gamma'_i$, $i=1,\ldots,n-1$, where $u_i$ involves only letters $\geq b$, and $v_i$ only letters $<b$. Moreover, by what has been said above, $\gamma_n=u_nv_n=u_nas_n$, $\gamma'_n=u_nbv'_n$, where $u_n$ involves only letters $>b$ and $v_n,v'_n$ only letters $<b$; moreover, $|v'_n|=|v_n|-1$. 
Then by the lemma, 
the column representatives of $y^nw$ and $y^nw'$ are respectively
$$
(yv_1)\cdots(yv_n) \cdots
$$
and 
$$
(yv'_1)\cdots(yv'_n) \cdots
$$
Then the $n$-th column of $P(y^nw)$ is longer than the $n$-th column of $P(y^nw')$. Thus $\lambda(y^nw)\neq \lambda (y^nw')$, and we take $x=y^n$.
 \end{proof}


\medskip

{\bf Acknowledgments} 

Special thanks are due to Jean-Eric Pin. He let run his software {\it Semigroupe} \cite{FP} on 
some examples we sent to him, and the tables he gave to us not only led us to the cardinality conjecture, but also to 
useful data on the stylic monoid. Note that hand calculations are almost impossible, except when $A$ is very small, since the number of states of the automaton for $|A|=n$ is $2^n$, the number of columns of $A$. His help was crucial.

We thank also Hugh Thomas for a useful mail exchange on the action of the plactic monoid on rows and columns.

We specially thank Darij Grinberg, who provided us with a long list of typos, and sloppiness in some proofs, and communicated us a new simpler proof of Proposition \ref{increasing} (ii), using Lemma \ref{LL} due to him.

Finally, we than the anonymous referee; he read thoroughly the two versions of this article, gave us a long list of typos, suggested changes, and found some gaps in proofs (in particular in the proof of Proposition \ref{zero}) and algorithms.

This work was partially supported by NSERC, Canada.

\end{document}